\documentclass{amsart}

\usepackage{amssymb, amscd,txfonts,stmaryrd,setspace,enumerate,makecell,wasysym}
\usepackage[usenames,dvipsnames]{xcolor}
\usepackage[bookmarks=true,colorlinks=true, linkcolor=MidnightBlue, citecolor=cyan]{hyperref}
\usepackage{lmodern}
\usepackage{graphicx,float}
\usepackage{tikz}
\usepackage{multirow}
\usetikzlibrary{matrix,arrows}

\input xy
\xyoption{all} \xyoption{poly} \xyoption{knot}\xyoption{curve}
\input{diagxy}

\newcommand{\comment}[1]{}

\def\tn{\textnormal}

\def\mc{\mathcal}

\def\ZZ{{\mathbb Z}}

\def\RR{{\mathbb R}}

\def\NN{{\mathbb N}}

\def\Hom{\tn{Hom}}

\def\Ob{\tn{Ob}}

\def\SEL*{\tn{SEL*}}
\def\Res{\tn{Res}}
\def\hsp{\hspace{.3in}}
\newcommand{\hsps}[1]{{\hspace{2mm} #1\hspace{2mm}}}
\newcommand{\tin}[1]{\text{\tiny #1}}

\newcommand{\singlefun}[1]{\ul{1}^{#1}}
\newcommand{\pullb}[1]{\Delta_{#1}}
\newcommand{\lpush}[1]{\Sigma_{#1}}
\newcommand{\rpush}[1]{\Pi_{#1}}
\def\lcone{^\triangleleft}

\def\to{\rightarrow}

\def\down{\downarrrow}
\def\Down{\Downarrow}
\def\cross{\times}
\def\taking{\colon}
\def\inj{\hookrightarrow}

\def\too{\longrightarrow}
\newcommand{\xyright}[1]{\xymatrix{~\ar[r]#1&}}
\newcommand{\xydown}[1]{\xymatrix{~\ar[d]#1\\~}}
\newcommand{\xydoown}[1]{\xymatrix{~\ar[ddd]#1\\\parbox{0in}{~}\\\parbox{0in}{~}\\~}}

\def\ss{\subseteq}

\def\iso{\cong}
\def\down{\downarrow}
\def\|{{\;|\;}}
\def\m1{{-1}}
\def\op{^\tn{op}}

\def\la{\langle}
\def\ra{\rangle}
\def\wt{\widetilde}

\def\ol{\overline}
\def\ul{\underline}

\def\del{\partial\,}

\newcommand{\LMO}[1]{\stackrel{#1}{\bullet}}
\newcommand{\LTO}[1]{\stackrel{\tt{#1}}{\bullet}}
\newcommand{\LA}[2]{\ar[#1]^-{\tn {#2}}}
\newcommand{\LAL}[2]{\ar[#1]_-{\tn {#2}}}
\newcommand{\obox}[3]{\stackrel{#1}{\fbox{\parbox{#2}{#3}}}}

\newcommand{\fakebox}[1]{\tn{$\ulcorner$#1$\urcorner$}}

\def\ullimit{\ar@{}[rd]|(.3)*+{\lrcorner}}
\def\urlimit{\ar@{}[ld]|(.3)*+{\llcorner}}
\def\lllimit{\ar@{}[ru]|(.3)*+{\urcorner}}
\def\lrlimit{\ar@{}[lu]|(.3)*+{\ulcorner}}
\def\ulhlimit{\ar@{}[rd]|(.3)*+{\diamond}}
\def\urhlimit{\ar@{}[ld]|(.3)*+{\diamond}}
\def\llhlimit{\ar@{}[ru]|(.3)*+{\diamond}}
\def\lrhlimit{\ar@{}[lu]|(.3)*+{\diamond}}
\newcommand{\clabel}[1]{\ar@{}[rd]|(.5)*+{#1}}
\newcommand{\TriRight}[7]{\xymatrix{#1\ar[dr]_{#2}\ar[rr]^{#3}&&#4\ar[dl]^{#5}\\&#6\ar@{}[u] |{\Longrightarrow}\ar@{}[u]|>>>>{#7}}}
\newcommand{\TriLeft}[7]{\xymatrix{#1\ar[dr]_{#2}\ar[rr]^{#3}&&#4\ar[dl]^{#5}\\&#6\ar@{}[u] |{\Longleftarrow}\ar@{}[u]|>>>>{#7}}}
\newcommand{\TriIso}[7]{\xymatrix{#1\ar[dr]_{#2}\ar[rr]^{#3}&&#4\ar[dl]^{#5}\\&#6\ar@{}[u] |{\Longleftrightarrow}\ar@{}[u]|>>>>{#7}}}

\newcommand{\arr}[1]{\ar@<.5ex>[#1]\ar@<-.5ex>[#1]}
\newcommand{\arrr}[1]{\ar@<.7ex>[#1]\ar@<0ex>[#1]\ar@<-.7ex>[#1]}
\newcommand{\arrrr}[1]{\ar@<.9ex>[#1]\ar@<.3ex>[#1]\ar@<-.3ex>[#1]\ar@<-.9ex>[#1]}
\newcommand{\arrrrr}[1]{\ar@<1ex>[#1]\ar@<.5ex>[#1]\ar[#1]\ar@<-.5ex>[#1]\ar@<-1ex>[#1]}

\newcommand{\To}[1]{\xrightarrow{#1}}
\newcommand{\Too}[1]{\xrightarrow{\ \ #1\ \ }}
\newcommand{\From}[1]{\xleftarrow{#1}}

\newcommand{\Adjoint}[4]{\xymatrix@1{#2 \ar@<.5ex>[r]^-{#1} & #3 \ar@<.5ex>[l]^-{#4}}}

\def\id{\tn{id}}
\def\Top{{\bf Top}}
\def\Cat{{\bf Cat}}

\def\Type{{\bf Type}}
\def\Set{{\bf Set}}
\def\Qry{{\bf Qry}}
\def\set{{\text \textendash}{\bf Set}}

\def\bD{{\bf \Delta}}
\def\dispInt{\parbox{.1in}{$\int$}}
\def\bhline{\Xhline{2\arrayrulewidth}}
\def\bbhline{\Xhline{2.5\arrayrulewidth}}
\def\bbbhline{\Xhline{3\arrayrulewidth}}

\def\colim{\mathop{\tn{colim}}}
\def\hocolim{\mathop{\tn{hocolim}}}

\def\mcC{\mc{C}}

\def\mcG{\mc{G}}

\def\mcX{\mc{X}}

\def\undsc{\rule{2mm}{0.4pt}}

\newtheorem{theorem}{Theorem}[subsection]
\newtheorem{lemma}[theorem]{Lemma}
\newtheorem{proposition}[theorem]{Proposition}
\newtheorem{corollary}[theorem]{Corollary}

\theoremstyle{remark}
\newtheorem{remark}[theorem]{Remark}
\newtheorem{example}[theorem]{Example}

\newtheorem{question}[theorem]{Question}
\newtheorem{guess}[theorem]{Guess}

\theoremstyle{definition}
\newtheorem{definition}[theorem]{Definition}

\def\Prb{{\bf Prb}}
\def\Prbs{{\wt{\bf Prb}}}

\def\Sch{{\bf Sch}}

\makeatletter\let\c@figure\c@equation\makeatother

\newcommand{\mainCatLarge}[1]{ 
	\stackrel{#1}{
		\parbox{4.5in}{\fbox{\parbox{4.4in}{\begin{center}\underline{{\tt Employee} manager worksIn $\simeq$ {\tt Employee} worksIn}\hsp  \underline{{\tt Department} secretary worksIn $\simeq$ {\tt Department}}\end{center}~\\\\\\
			\xymatrix@=8pt{&\LTO{Employee}\ar@<.5ex>[rrrrr]^{\tn{worksIn}}\ar@(l,u)[]+<5pt,10pt>^{\tn{manager}}\ar[dddl]_{\tn{first}}\ar[dddr]^{\tn{last}}&&&&&\LTO{Department}\ar@<.5ex>[lllll]^{\tn{secretary}}\ar[ddd]^{\tn{name}}\\\\\\\LTO{FirstNameString}&&\LTO{LastNameString}&~&~&~&\LTO{DepartmentNameString}
			}
		}}}
	}
}

\begin{document}

\title{Database queries and constraints via lifting problems}

\author{David I. Spivak}

\address{Department of Mathematics, Massachusetts Institute of Technology, Cambridge MA 02139}

\email{dspivak@math.mit.edu}

\thanks{This project was supported by ONR grant N000141010841.}

\begin{abstract}

Previous work has demonstrated that categories are useful and expressive models for databases. In the present paper we build on that model, showing that certain queries and constraints correspond to lifting problems, as found in modern approaches to algebraic topology. In our formulation, each so-called SPARQL graph pattern query corresponds to a category-theoretic lifting problem, whereby the set of solutions to the query is precisely the set of lifts. We interpret constraints within the same formalism and then investigate some basic properties of queries and constraints. In particular, to any database $\pi$ we can associate a certain derived database $\Qry(\pi)$ of queries on $\pi$. As an application, we explain how giving users access to certain parts of $\Qry(\pi)$, rather than direct access to $\pi$, improves ones ability to manage the impact of schema evolution.

\end{abstract}

\maketitle

\tableofcontents

\section{Introduction}\label{sec:intro}

In \cite{DK}, \cite{JoM}, \cite{JRW}, and many others, a tight connection between database schemas and the category-theoretic notion of sketches was presented and investigated. This connection was carried further in \cite{Sp2} where the existence of three data migration functors was shown to follow as a simple consequence of using categories rather than sketches to model schemas. In this paper we shall show that a modern approach to the study of algebraic topology, the so-called {\em lifting problem} approach (see \cite{Qui}), provides an excellent model for typical queries and constraints (see \cite{PS}).

A database consists of a schema (a layout of tables in which so called {\em foreign key} columns connect one table to another) and an instance (the rows of actual data conforming to the chosen layout). One can picture the analogy between databases and topological spaces as follows. Imagine that a collection of data $I$ and a schema $S$ are each an abstract space, and suppose we have a projection from $I$ to $S$. That is, we have some kind of continuous map $\pi\taking I\to S$ from a ``data bundle" $I$ to a ``base space" $S$. Points in $S$ represent tables, and paths in $S$ represent foreign key columns (or iterates thereof), which point from one table to another. Over every point $s\in S$ in the base space, we can look at the corresponding fiber $\pi^{-1}(s)\ss I$ of the data bundle; this will correspond to the set of rows in table $s$. The map $\pi$, associating data with schema, is called a {\em database instance}.

A {\em query} on a database instance $\pi\taking I\to S$ is like a system of equations: it includes an organized collection of knowns and unknowns. In our model a query takes the form of a functor $m\taking W\to R$, such that $W$ (standing for WHERE-clause) corresponds to the set of knowns, each of which maps to a specific value in the data bundle $I$, and such that the relationship between knowns and unknowns is captured in a schema $R$. More precisely, a query on the database instance $\pi\taking I\to S$ is presented as a commutative diagram to the left, which would be roughly translated into the pseudo-SQL to the right,\footnote{A more general SQL query, with a specific SELECT statement will be discussed in Example \ref{ex:SQL statement}.} in (\ref{dia:lifting SQL}): 
\begin{align}\label{dia:lifting SQL}
\parbox{.8in}{\xymatrix{W\ar[r]^p\ar[d]_m&I\ar[d]^\pi\\R\ar[r]_n&S.}}
\hspace{1in}
\parbox{1.4in}{
\begin{tabbing}
SELECT\;\; \=$\ast$\\
FROM \>$R\To{n}S$\\
WHERE \>$R\From{m}W\To{p}I$
\end{tabbing}}
\end{align} 
A {\em result} to the query is any mapping $\ell\taking R\to I$ making both triangles commute ($\ell\circ m=p, \pi\circ\ell=n$) in the diagram \begin{align}\label{dia:intro lifting}\xymatrix{W\ar[r]^p\ar[d]_m&I\ar[d]^\pi\\R\ar[ur]^\ell\ar[r]_n&S.}\end{align} The map $\ell$ is called a {\em lift} of Diagram \ref{dia:lifting SQL}, hence the term {\em lifting problem}. The idea is that a lift is a way to fill the result schema $R$ with conforming data from the instance $\pi$.

We will now give a simple example from algebraic topology to strengthen the above image. By connecting databases and topology, we not only can visualize queries in a new way, but it is conceivable that algebraic topologists could use database interfaces to have computers work on lifting problems that arise in their research. Regardless, after the topological example, we will ground the discussion with an example database query.

Consider an empty sphere, defined by the equation $x^2+y^2+z^2=1$; call it $I$. We project it down onto the $(x,y)$-coordinate plane ($z=0$); call that plane $S\iso\RR^2$. The sphere $I$ serves as the database instance and the plane $S$ serves as the schema. A query consists of some result schema mapping to the plane $S$, say a solid disk $R$ (given by $z=0, x^2+y^2\leq1$), together with some constraints, say on the boundary circle $W$ (given by $z=0, x^2+y^2=1$) of the disk. Graphically we have Figure \ref{sphere-top}.

\begin{flushleft}\setcounter{figure}{\value{equation}}
\begin{figure}
\caption{A topological lifting problem} 
\label{sphere-top}
\includegraphics[height=9.5cm]{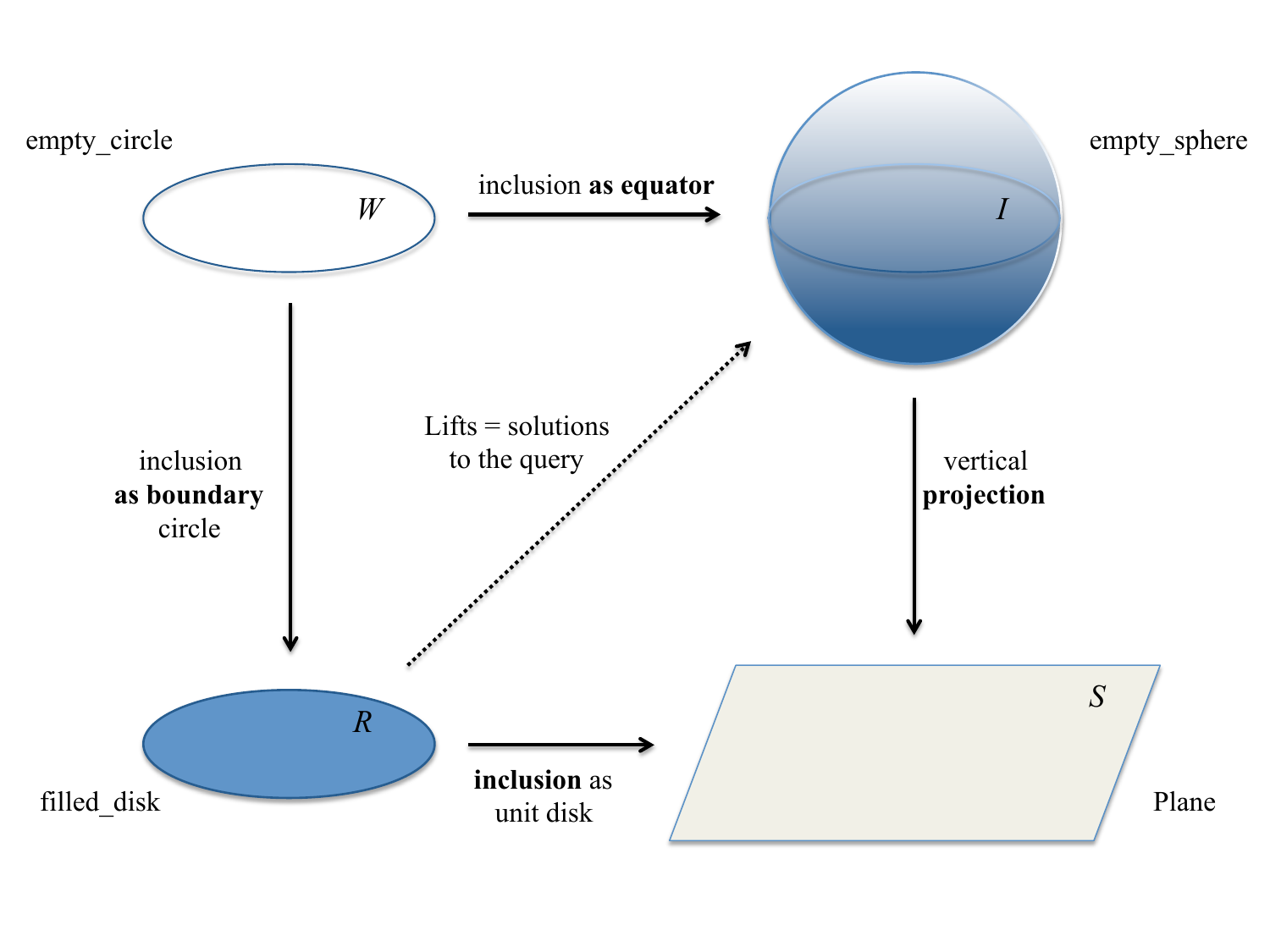} 
\end{figure}
\end{flushleft} 
The results of the lifting query from Figure \ref{sphere-top} are the mappings $R\to I$ making the diagram commute. Under the guidance of (\ref{dia:lifting SQL}) the query would look something like this: 
\begin{tabbing}\label{dia:SQL of sphere-top}
\hspace{.52in}\=SELECT \=$\ast$\\
\>FROM \>filled\undsc disk inclusion\\
\>WHERE \>empty\undsc circle as boundary\;=\;empty\undsc circle as equator
\end{tabbing}
Topologically one checks that there are exactly two lifts---the top hemisphere and the bottom hemisphere---so our pseudo-SQL query above would return exactly two results.

\subsection{Main example of a lifting query}\label{sec:main example}

We now provide an example of a situation in which one may wish to query a database, and we show that this query naturally takes the structure of a lifting problem. We break a single example into three parts for clarity.

\begin{example}[Main Example 1: Situation, SPARQL, and schema]\label{ex:query party setup}

Suppose you have just come home from a party. There, you met and really hit it off with a married couple; the husband's name is Bob and the wife's name is Sue; they live in Cambridge. From your conversation, you know that Bob works at MIT and Sue works in the financial sector. You'd like to see them again, but you somehow forgot to ask for their contact information; in particular you'd like to know their last names. 

This is a typical database query problem. It can be phrased as the following SPARQL graph pattern query (which we arrange in two columns for space and readability reasons):
 \begin{align}\label{dia:SPARQL}
 \begin{tabular}{lll}
(?marriage includesAsHusband ?b) &\hsp& (?marriage includesAsWife ?s)\\
(?b hasFirstName Bob)&&(?s hasFirstName Sue)\\
(?b livesIn Cambridge)&&(?s livesIn Cambridge)\\
(?employedb is ?b)&&(?employeds is ?s)\\
(?employedb hasEmployer MIT)&&(?employeds hasEmployer ?sueEmp)\\
&&(?sueEmp isIn financial)\\
(?b hasLastName ?bobLast)&&(?s hasLastName ?sueLast)
\end{tabular}\end{align}

The query in (\ref{dia:SPARQL}) might be asked on the following database schema:\footnote{The schema $S$ in (\ref{dia:schema bobsue}) deliberately includes a box $D$ and an arrow $G\to D$ that are not part of our query (\ref{dia:SPARQL}).}
\begin{align}\label{dia:schema bobsue}
S:=\parbox{3.5in}{
\fbox{\xymatrix{
&\obox{G}{.7in}{a marriage}\ar[r]^-{\parbox{.5in}{\singlespace\scriptsize\begin{center} had wedding on\end{center}}}\LAL{dl}{includes as husband\hsp}\LA{dr}{\hsp includes as wife}&\obox{D}{.4in}{a date}\\
\obox{M}{.5in}{a man}\LA{dr}{is}&&\obox{W}{.6in}{a woman}\LAL{dl}{is}\\
\obox{E}{.8in}{an employed person}\LA{r}{is}\LA{d}{has employer}&\obox{P}{.6in}{a person}\LA{d}{lives in}\LA{r}{\parbox{.5in}{\singlespace\scriptsize\begin{center}has\\first name\end{center}}}\LA{dr}{\hspace{.15in}has last name}&\obox{F}{.7in}{a first name}\\
\obox{Y}{.8in}{an employer}\LA{d}{is in}&\obox{C}{.4in}{a city}&\obox{L}{.7in}{a last name}\\
\obox{T}{.6in}{a sector}}}}
\end{align} Given that $S$ is instantiated with data $\pi\taking I\to S$, one can hope to find Bob and Sue, and then determine their last name. In the following two examples (Examples \ref{ex:query party 1} and \ref{ex:query party 2}) we will show that this query corresponds to a lifting problem for $\pi$. 

\end{example}

\begin{example}[Main Example 2: WHERE-clause and Result schema]\label{ex:query party 1}

Recall the SPARQL query presented as (\ref{dia:SPARQL}) in Example \ref{ex:query party setup}, in which we wanted to find information about our new friends Bob and Sue. We will use a lifting problem to state this query; to do so we need to come up with a result schema $R$, a constraint schema (a set of knowns) $W$, and a mapping $m\taking W\to R$ embedding the known objects into the result schema. In this example we will present $m, W,$ and $R$. In Example \ref{ex:query party 2} we will explain the lifting diagram for the query and show the results. 

In order to find our friends Bob and Sue, we will use the following mapping: 
\begin{align*}
\tiny\parbox{5in}{
\begin{center}W:=\end{center}
\fbox{\xymatrix@=14pt{\obox{Y1}{.2in}{MIT}&~&&\hspace{.15in}&&\obox{T}{.4in}{financial sector}\\~&~&&&&~&~\\\\\obox{F1}{.2in}{Bob}&\hsp&\obox{C1}{.5in}{Cambridge}&~&\obox{C2}{.5in}{Cambridge}&\hsp&\obox{F2}{.2in}{Sue}}}
}\\
\parbox{2.6in}{$\xymatrix{~\ar[d]^m\\~}$}\\
\tiny\parbox{5in}{
\begin{center}R:=\end{center}
\fbox{\xymatrix@=14pt{\obox{Y1}{.3in}{an employer}&&&\obox{G}{.3in}{a marriage}\LAL{dll}{includes as husband\hsp}\LA{drr}{\hsp includes as wife}&&\obox{T}{.4in}{a sector}&\obox{Y2}{.3in}{an employer}\LAL{l}{is in}\\\obox{E1}{.3in}{an employed person}\LA{r}{is}\LAL{u}{has}&\obox{P1}{.3in}{a person}\LA{d}{has}\LA{dr}{lives in}\LAL{dl}{has}&&&&\obox{P2}{.3in}{a person}\LA{d}{has}\LA{dr}{has}\LAL{dl}{lives in}&\obox{E2}{.3in}{an em-ployed person}\LAL{l}{is}\LAL{u}{has}\\\obox{F1}{.3in}{a first name}&\obox{L1}{.3in}{a last name}&\obox{C1}{.3in}{a city}&&\obox{C2}{.3in}{a city}&\obox{L2}{.3in}{a last name}&\obox{F2}{.3in}{a first name}}}
}
\end{align*} 
The functor $m\taking W\to R$ is indicated by sending each object in $W$ to the object with the same label in $R$; e.g. \fakebox{MIT} in $\Ob(W)$ is sent to \fakebox{an employer} in $\Ob(R)$ because they are both labeled $Y1$. 

To orient oneself, we suggest the following. Count the number of constants in the SPARQL query (\ref{dia:SPARQL})---there are 6 (such as Bob, Cambridge, etc.); this is precisely the number of objects in $W$. Count the combined number of constants and variables in the SPARQL query---there are 14 (there are 8 variables, such as ?marriage, ?empoyedb, etc.); this is precisely the number of objects in $R$. Finally, count the number of triples in the SPARQL query -- there are 13; this is precisely the number of arrows in $R$. These facts are not coincidences.

\end{example}

\begin{example}[Main Example 3: Lifting diagram and result set]\label{ex:query party 2}

In Example \ref{ex:query party 1} we showed a functor $m\taking W\to R$ corresponding to the SPARQL query stated in (\ref{dia:SPARQL}). In this example we will explain how this query can be formulated as a lifting problem of the form \begin{align}\label{dia:qp2 lifting}\xymatrix{W\ar[r]^p\ar[d]_m&I\ar[d]^\pi\\R\ar@{-->}[ur]^\ell\ar[r]_n&S}\end{align} which serves to pose our query to the database instance $\pi$. At this point we can ask for the set of solutions $\ell$. So far, $W,m,R,$ and $S$ have been presented, $I$ and $\pi$ have been assumed, and the set of $\ell$'s is coming later, so it suffices to present $p$ and $n$. 

One should refer to our presentation of $S$ in Example \ref{ex:query party setup} (\ref{dia:schema bobsue}). The functor $n\taking R\to S$ should be obvious from our labeling system (for example, the object E1=\fakebox{an employed person} in category $R$ is mapped to the object E=\fakebox{an employed person} in category $S$). Note that, as applied to objects, $n$ is neither injective nor surjective in this case: $n^\m1(P)=\{P_1,P_2\}$ and $n^\m1(D)=\emptyset$.

Suppose $\pi\taking I\to S$ is our data bundle, and assume that it contains enough data that the constants in the query have unique referents.
\footnote{This use of the term ``query" is not standard. See Sections \ref{sec:queries on database} and \ref{sec:binding variables} for an explanation.}
There is an obvious functor $p\taking W\to I$ that sends each object in category $W$ to its referent in $I$. For example, we assume that there is an object in $I$ labelled \fakebox{MIT}, which is mapped to by the object Y1=\fakebox{MIT} in $W$.

Thus our query from (\ref{dia:SPARQL}) is finally in the form of a lifting problem as in (\ref{dia:qp2 lifting}). We will show in Example \ref{ex:bob and sue revisited}, after we have built up the requisite theory, that the set of lifts can be collected into a single table, the most useful projection of which would look something like this: 
\begin{align}\label{dia:result state}
\footnotesize
\begin{tabular}{| c ||| c || c | c | c | c || c | c | c |}
\bhline
\multicolumn{9}{| c |}{\bf Marriage}\\
\bhline 
\multirow{2}{*}{\bf ID}&\multicolumn{4}{ c |}{\bf Husband}&\multicolumn{4}{ c |}{\bf Wife}\\\cline{2-9}\cline{2-9}
&{\bf ID}&{\bf First}&{\bf Last}&{\bf City}&{\bf ID}&{\bf First}&{\bf Last}&{\bf City}\\
\bbbhline G3801&M881-36&Bob&Graf&Cambridge&W913-55&Sue&Graf&Cambridge\\
\bhline
\end{tabular}
\end{align}

This concludes the tour of our main example: we have shown a typical query formulated as a lifting problem. The mathematical basis for the above ideas will be presented in Section \ref{sec:queries as LPs}.

\end{example}

\subsection{Relation to earlier work}

As mentioned above, there is a long history of applying category-theoretic formalism to database theory. These approaches can roughly be broken into two schools. The first school, including \cite{Tui} and \cite{Kat}, considers relational database tables as sets of attributes, using limits to discuss joins. This approach is similar to that used in \cite{Sp1}, in which simplicial sets were used as a geometric model for ``sheaves of attributes". The second school, including \cite{DK}, \cite{JoM}, and \cite{JRW}, uses so-called sketches in the sense of \cite{Ehr}. The latter approach is closer to that in \cite{Sp1} and the present article. We now discuss the key differences between this approach and the present one.

As we will discuss in Section \ref{sec:review ctdb}, we model database schemas as categories, whereas the second school above models them as sketches. A sketch is a category together with specified limit cones and colimit cones. Sketches are more expressive than categories, for example allowing schemas to convey when the set of rows in table $T$ is the product of the sets of rows in tables $U$ and $V$. This expressivity comes at a cost: whereas the categorical model in \cite{Sp2} has three built in data migration functors corresponding to project, union, and join of queries, the sketch model has only one: project. In other words, being able to specify limits and colimits in a given schema limits ones ability to translate data given a morphism of schemas, either in the case of schema evolution or in the case of comparison with a schema built by another group. 

Still, it may be useful to find something in between sketches and bare categories, because using categories as models does not allow one to express constraints beyond foreign keys and commutative diagrams. For example it does not allow for injectivity, or ``is a", constraints. It is here that the present paper fits in. Modern mathematical research, especially algebra and topology, has found very little use for sketches and sketch morphisms, whereas it has become deeply invested in categories and functors. Further, algebraic topology, the trailblazer for category theory, has for more than half a century found lifting problems to be a key tool for investigating abstract spaces. In this paper we make the connection between these ideas and databases. As mentioned above, we show that there are many constraints that are well-phrased as lifting problems, and that queries also fit nicely into this framework. 

Sketches are often divided into two types: limit sketches and colimit sketches. Lifting constraints fully cover the expressivity of limit sketches and more; see Section \ref{sec:lifting more expressive}. In particular, lifting problems can enforce injectivity constraints,  as shown in Example \ref{ex:const inj}. However, colimit sketches can express things than lifting constraints cannot. For example, with colimit sketches one can express set-theoretic complements, and this cannot be done with lifting problems. The ability to enforce that one subset is the complement of another comes with well-known problems. The point is that lifting problems can express a different class of constraints than that expressible by sketches, and Section \ref{sec:examples} is designed to show that the set of constraints expressible by lifting problems may be useful.

\subsection{Purpose of the paper}

The purpose of this paper is to:
\begin{itemize}
\item provide an efficient mathematical formulation of common database queries (modeling both SQL and SPARQL styles),
\item attach a geometric image to database queries that can be useful in conceptualization, and
\item explore theory and applications of the derived database schema $\Qry(\pi)$ of queries on a database instance $\pi$, and the derived instance of results.
\end{itemize}  

We include several mathematical results that are well-known to experts, for the purpose of aiding those interested in using this paper to bridge the gap between database theory and category theory. 

\subsection{Plan of the paper}

We begin in Section \ref{sec:review} with a review of the categorical approach to databases (see \cite{Sp2} for more details). Roughly this correspondence goes by the following slogan: ``schemas are categories, instances are set-valued functors". In Section \ref{sec:discrete opfibrations} we also discuss the Grothendieck construction, which will be crucial for our approach: a database instance can be converted into a so-called discrete opfibration, which we will later use extensively to make the parallel with algebraic topology and lifting problems in particular. 

In Section \ref{sec:constraints} we define constraints on a database in terms of lifting conditions and discuss some constraint implications. We give several examples to show how various common existence and uniqueness constraints (such as the constraint that a given foreign key column is surjective) can be framed in the language of lifting conditions. In Section \ref{sec:queries as LPs} we discuss queries as lifting problems, and review the paper's main example. In Section \ref{sec:formal properties}, we show that the information in a given database instance can be collected into a new, derived database. This derived database of queries and their results can be queried, giving rise to nested queries. We explain how this formulation can be useful for managing the impact of schema evolution. Finally in Section \ref{sec:future} we briefly discuss some possible directions for future work, including tying in to Homotopy Type Theory (in the sense of \cite{Awo} and \cite{Voe}) and other projects.

\subsection{Notation and terminology}\label{sec:notation}

For any natural number $n\in\NN$, let $\ul{n}$ denote the set $\{1,2,\ldots,n\}$. We sometimes regard sets as discrete categories without mentioning it. Note that $\ul{0}=\emptyset$. Let $[n]$ denote the linear order $0\leq 1\leq\ldots\leq n$. We sometimes regard orders as categories without mentioning that either. In particular $\ul{1}$ is the terminal category; it has one object and one morphism (the identity). 

Given any category $\mcC$, we denote the category of all functors $\mcC\to\Set$ by $\mcC\set$. The terminal object in $\mcC\set$ sends each object in $\mcC$ to $\ul{1}$; we denote it by $\ul{1}^{\mcC}\taking\mcC\to\Set$. For any category $\mcC$, there is a one-to-one correspondence between the objects in $\mcC$ and the functors $\ul{1}\to\mcC$. Thus we may denote an object $c\in\Ob(\mcC)$ by a functor $\ul{1}\To{c}\mcC$. In particular, we elide the difference between a set and a functor $\ul{1}\to\Set$. 

We draw schemas in one of two ways. When trying to save space, we draw our objects as concisely-labeled nodes and our morphisms as concisely-labeled arrows; when trying to be more expressive, we draw our objects as text boxes and put as much text in them (and on each arrow) as is necessary to be clear (see \cite{SK}). For example, we might draw the indexing category for directed graphs in either of the following two ways: $$\fbox{\xymatrix{\LMO{E}\ar@/^1pc/[rr]^s\ar@/_1pc/[rr]_t&&\LMO{V}}}\hsp\hsp\fbox{\xymatrix{\obox{}{.5in}{an edge}\ar@/^1pc/[rr]^{\tn{has as source}}\ar@/_1pc/[rr]_{\tn{has as target}}&&\obox{}{.5in}{a vertex}}}$$ When in the typographical context of inline text we are discussing an object that has been elsewhere displayed as a textbox (such as \fbox{an edge}), we may represent it with corner symbols (e.g. as \fakebox{an edge}) to avoid various spacing issues that can arise.

Given two categories, there are generally many functors from one to the other; however, if the objects and arrows are labeled coherently, there are many fewer functors that roughly respect the labelings. We will usually be explicit when defining functors, but we will also take care that our functors respect labeling to the extent possible. 

\subsubsection{``Queries on a database"}\label{sec:queries on database}

In wide-spread terminology for database queries, a query cannot depend on the current instance $\pi$ of the database, but instead only on the schema $S$. This is perfectly reasonable for theoretical and practical reasons. Often in applications, however, one uses what is known as a {\em cursor}, which is basically a pre-defined query consisting of a join-graph and a set of variables to be bound at run-time. With respect to the diagram
$$
\xymatrix{W\ar[r]^p\ar[d]_m&I\ar[d]^\pi\\R\ar@{-->}[ur]^\ell\ar[r]_n&S.}
$$
the join-graph is $R$, the set of variables waiting to be bound is $W$, and the binding itself is $p\taking W\to I$. The mathematics will be covered more extensively in Section \ref{sec:binding variables}; in the remaining paragraphs of Section \ref{sec:queries on database}, we hope to get across how one might connect our use of the term ``query" in the present paper to common ideas in database systems.

In applications, a query wizard may run the cursor in a 2-step query process: first it will query the database to offer the user a drop-down menu of choices in the active domain of each variable. The user will choose a row to which the variable will be bound (once for each variable). At this point the program will apply the actual query declared by the cursor. This two step process corresponds to searching for possible functors $p\taking W\to I$ and then searching for lifts $\ell$. 

Throughout this article, when we speak of queries on a database, we mean queries for which the constant variables have been bound to elements in the active domain of a given instance. However, as we will see in Section \ref{sec:binding variables}, one can also use the same machinery in cursor-like fashion to pose queries in which variable values have been chosen without regard for whether or not they are in the active domain. In other words, we will see that what can be accomplished by queries in the sense of traditional relational database theory fits easily into our framework. Because it works either way, we thought that the unusual terminology ``queries on a database" would be best because it neither lulls the reader into thinking that these gadgets are completely instance-independent, nor frightens the reader into thinking that the instance must be known in advance for the ideas here to work. 

\subsection{Acknowledgments}

I would like to thank Peter Gates, as well as Henrik Forssell, Rich Haney, Eric Prud'hommeaux, and Emily Riehl for many useful discussions.

\section{Elementary theory of categorical databases}\label{sec:review}

\subsection{Review of the categorical description of databases}\label{sec:review ctdb}

The basic mantra is that a database schema is a small category $S$ and an instance is a functor $\delta\taking S\to\Set$, where $\Set$ is the category of sets.\footnote{If one prefers, $\Set$ can be replaced by the category of finite sets or by the category Types for some $\lambda$-calculus.} To recall these ideas, we take liberally from \cite{Sp2}, though more details and clarification are given there. Readers who are familiar with the basic setup and data migration functors can skip to Section \ref{sec:discrete opfibrations}.

In \cite{Sp2} a category $\Sch$ of categorical schemas and translations is defined and an equivalence of categories 
\begin{align}\label{dia:equivalence Sch Cat}
\Sch\simeq\Cat
\end{align} 
is proved, where $\Cat$ is the category of small categories. The difference between $\Sch$ and $\Cat$ is that an object of the former is a {\em chosen presentation} of a category, by generators and relations, as described below. Given the equivalence (\ref{dia:equivalence Sch Cat}), we can and do elide the difference between schemas and small categories.

Roughly, a schema $S$ consists of a graph $G$ together with an equivalence relation on the set of paths of $G$. Each object $s\in\Ob(S)$ represents a table (or more precisely the ID column of a table), and each arrow $s\to t$ emanating from $s$ represents a column of table $s$, taking values in the ID column of table $t$. An example should clarify the ideas.

\begin{example}

As a typical database example, consider the bookkeeping necessary to run a department store. We keep track of a set of employees and a set of departments. For each employee $e$, we keep track of
\begin{enumerate}[\hsp E.1\;]
\item the {\bf first} name of $e$, which is a {\tt FirstNameString},
\item the {\bf last} name of $e$, which is a {\tt LastNameString},
\item the {\bf manager} of $e$, which is an {\tt Employee}, and
\item the department that $e$ {\bf works in}, which is a {\tt Department}.
\end{enumerate}
For each department $d$, we keep track of
\begin{enumerate}[\hsp D.1\;]
\item the {\bf name} of $d$, which is a {\tt DepartmentNameString}, and
\item the {\bf secretary} of $d$, which is an {\tt Employee}.
\end{enumerate}
Suppose further that we make the following two rules. 
\begin{enumerate}[\hsp Rule 1\;]
\item For every employee $e$, the {\bf manager} of $e$ {\bf works in} the same department that $e$ {\bf works in}.
\item For every department $d$, the {\bf secretary} of $d$ {\bf works in} department $d$.
\end{enumerate}

This is all captured neatly, with nothing left out and nothing else added, by the category presented below:
\begin{align}\label{dia:basic cat}S:=\mainCatLarge{}\end{align}
The underlined statements at the top indicate pairs of commutative (i.e. equivalent) paths; each path is indicated by its source object followed by the sequence of arrows that composes it. The objects, arrows, and equivalences in $S$ correspond to the tables, columns, and rules laid out at the beginning of this example.

The collection of data on a schema is typically presented in table form. Display (\ref{dia:instance on maincat}) shows how a database with schema $S$ might look at a particular moment in time. 
\begin{align}\label{dia:instance on maincat}
&\footnotesize
\begin{tabular}{| l || l | l | l | l |}\bhline
\multicolumn{5}{| c |}{{\tt Employee}}\\\bhline 
{\bf ID}&{\bf first}&{\bf last}&{\bf manager}&{\bf worksIn}\\\bbhline 101&David&Hilbert&103&q10\\\hline 102&Bertrand&Russell&102&x02\\\hline 103&Alan&Turing&103&q10\\\bhline
\end{tabular}&\hsp\footnotesize
\begin{tabular}{| l || l | l |}\bhline
\multicolumn{3}{| c |}{{\tt Department}}\\
\bhline {\bf ID}&{\bf name}&{\bf secretary}\\\bbhline q10&Sales&101\\\hline x02&Production&102\\\bhline
\end{tabular}
\end{align}\vspace{.1in}
\begin{align*}\footnotesize
\begin{tabular}{| l ||}\bhline
\multicolumn{1}{| c |}{{\tt FirstNameString}}\\\bhline
{\bf ID}\\\bbhline Alan\\\hline Alice\\\hline Bertrand\\\hline Carl\\\hline David\\\hline\hspace{.25in}\vdots\\\bhline
\end{tabular}\hspace{.6in}\footnotesize
\begin{tabular}{| l ||}\bhline
\multicolumn{1}{| c |}{{\tt LastNameString}}\\\bhline
{\bf ID}\\\bbhline Arden\\\hline Hilbert\\\hline Jones\\\hline Russell\\\hline Turing\\\hline\hspace{.25in}\vdots\\\bhline
\end{tabular}\hspace{.6in}\footnotesize
\begin{tabular}{| l ||}\bhline
\multicolumn{1}{| c |}{{\tt DepartmentNameString}}\\\bhline
{\bf ID}\\\bbhline Marketing\\\hline Production\\\hline Sales\\\hline\hspace{.25in}\vdots\\\bhline
\end{tabular}
\end{align*}
Every table has an ID column, and in every table each cell references a cell in the ID column of some table. For example, cells in the secretary column of the {\tt Department} table refer to cells in the ID column of the {\tt Employee} table. Finally, one checks that Rule 1 and Rule 2 hold. For example, let $e$ be Employee 101. He works in Department q10 and his manager is Employee 103. Employee 103 works in Department q10 as well, as required. The point is that the data in (\ref{dia:instance on maincat}) conform precisely to the schema $S$ from Diagram (\ref{dia:basic cat}).

A set of tables that conforms to a schema is called an {\em instance} of that schema. Let us denote the set of tables from (\ref{dia:instance on maincat}) by $\delta$; we noted above that $\delta$ conforms with, thus is an instance of, schema $S$. Mathematically, $\delta$ can be modeled as a functor $$\delta\taking S\to\Set.$$ To each object $s\in S$ the instance $\delta$ assigns a set of row-IDs, and to each arrow $f\taking s\to t$ in $S$ it assigns a function, as specified by the cells in the $f$-column of $s$.

\end{example}

\subsection{Review of data migration functors}

Once we realize that a database schema can be captured simply as a category $S$ and each instance on $S$ as a set-valued functor $\delta\taking S\to\Set$, classical category theory gives ready-made tools for migrating data between different schemas. The first definition we need is that of schema mapping.

\begin{definition}

Let $S$ and $T$ be schemas (i.e. small categories). A {\em schema mapping} is a functor $F\taking S\to T$.

\end{definition}

Thus a schema mapping assigns to each table in $S$ a table in $T$, to each column in $S$ a column in the corresponding table of $T$, and all this in such a way that the path equivalence relation is preserved.

\begin{definition}\label{def:migration}

A schema mapping $F\taking S\to T$ induces three functors on instance categories, which we call {\em the data migration functors associated to $F$} and which we denote by $\lpush{F},\pullb{F},$ and $\rpush{F}$, displayed here: $$\xymatrix{S\set\ar@/^1pc/[rr]^{\lpush{F}}\ar@/_1pc/[rr]_{\rpush{F}}&&T\set.\ar[ll]|-{\pullb{F}}}$$ The functor $\pullb{F}\taking T\set\to S\set$ sends an instance $\delta\taking T\to\Set$ to the instance $\delta\circ F\taking S\to\Set$. The functor $\lpush{F}$ is the left adjoint to $\pullb{F}$, and the functor $\rpush{F}$ is the right adjoint to $\pullb{F}$. We call $\pullb{F}$ the {\em pullback along $F$} , we call $\lpush{F}$ the {\em left pushforward along $F$}, and we call $\rpush{F}$ the {\em right pushforward along $F$}. 

\end{definition}

The functors $\Delta_F,\Sigma_F,$ and $\Pi_F$ are well-known in category theory literature, where the latter two are often referred to as the {\em left Kan extension along $F$} and the {\em right Kan extension along $F$} (see \cite[X]{Mac}). In databases, the left pushforward $\Sigma_F$ will generally correspond to unions and quotients, and the right pushforward $\Pi_F$ will generally correspond to products and joins. We explore these ideas a bit further in Section \ref{sec:formal properties}; see \cite{Sp2} for further explanation.

\subsection{RDF via the Grothendieck construction}\label{sec:discrete opfibrations}

There is a well-known construction that associates to a functor $\delta\taking S\to\Set$, a pair $(\int(\delta),\pi_\delta)$ where $\int(\delta)\in\Cat$ is a new category, called {\em the category of elements of $\delta$}, and $\pi_\delta\taking\int(\delta)\to S$ is a functor. It is often called the {\em Grothendieck construction}. The objects and morphisms of $\int(\delta)$ are given as follows \begin{align*}\Ob(\dispInt(\delta))&:=\Big\{(s,x)\;|\;s\in\Ob(S),x\in\delta(s)\Big\}\\\Hom_{\int(\delta)}((s,x),(s',x'))&:=\Big\{f\taking s\to s'\;|\;\delta(f)(x)=x'\Big\}\end{align*} The functor $\pi_\delta\taking\int(\delta)\to S$ is straightforward: it sends the object $(s,x)$ to $s$ and sends the morphism $f\taking(s,x)\to(s',x')$ to $f\taking s\to s'$.

We call the pair $(\int(\delta),\pi_\delta)$ the {\em discrete opfibration associated to} $\delta$. We will see in the next section (Definition \ref{def:dofib}) that $\pi_\delta$ is indeed a kind of fibration of categories. This construction, and in particular the category $\int(\delta)$, is also nicely connected with the {\em resource descriptive framework} (see \cite{PS}), in which data is captured in so-called RDF triples. Indeed, the arrows $\LMO{s}\To{\;p\;}\LMO{b}$ of $\int(\delta)$ correspond one-for-one with these RDF triples ({\bf s}ubject, {\bf p}redicate, o{\bf b}ject). Thus we have shown a readymade conversion from relational databases to RDF triple stores via the Grothendieck construction. An example should clarify this discussion.

\begin{example}\label{ex:discrete opfibration}

Recall the database instance $\delta\taking S\to\Set$ given by the tables in Diagram (\ref{dia:instance on maincat}), whose schema $S$ was presented as Diagram (\ref{dia:basic cat}).  Applying the Grothendieck construction to $\delta\taking S\to\Set$, we get a category $I:=\int(\delta)$ and a functor $\pi:=\pi_\delta$ as in Figure \ref{fig:Grothendieck}.

\begin{figure}[h]
\hspace{0in}I=\parbox{3.9in}{
\fbox{
\xymatrix@=.1pt{
&\LTO{101}\ar@/_2.5pc/[ddddl]+<-3pt,3pt>_{\tin{first}}\ar@/_1.5pc/[ddrrr]+<-4pt,0pt>^{\tin{last}}\ar@/_1pc/[rr]+<-3pt,-3pt>_-{\tin{manager}}\ar@/^1.8pc/[rrrrrr]^{\tin{worksIn}}&\LTO{\;\;102}&\LTO{103}&&&&\LTO{q10}&\LTO{x02}\ar@/^1.6pc/[llllll]+<4pt,-2pt>_{\tin{secretary}}\ar@/^1.2pc/[lddd]+<6pt,2pt>^(.4){\tin{name}}\\\parbox{.1in}{~\\\vspace{.6in}~}\\\LTO{Alan}&&&&\LTO{Hilbert}&\hspace{.4in}&\hspace{.4in}&\LTO{Production}\\\LTO{\;\;Bertrand}&&&&\LTO{Russell}&&&\LTO{\hspace{-.1in}Sales}\\\LTO{\hspace{.2in}David}&&&&\LTO{Turing}&&&\LTO{Marketing}\\\LTO{Alice}&&&&\LTO{Arden}\\\LTO{Carl}&&&&\LTO{Smith}}}\\
\xymatrix{\hspace{1.6in}&\ar[d]^\pi\\&~}}\\
\nonumber S=\parbox{3.75in}{\hspace{.1in}\fbox{
			\xymatrix@=9pt{&\LTO{Employee}\ar@<.5ex>[rrrrr]^{\tn{worksIn}}\ar@(l,u)[]+<5pt,10pt>^{\tn{manager}}\ar[dddl]_{\tn{first}}\ar[dddr]^{\tn{last}}&&&\hspace{.3in}&&\LTO{Department}\ar@<.5ex>[lllll]^{\tn{secretary}}\ar[ddd]^{\tn{name}}\\\\\\\LTO{FNString}&&\LTO{LNString}&~&~&~&\LTO{DNString}
			}}}
\caption{An example of the Grothendieck construction, or category of elements, of a functor $\delta\taking S\to\Set$.. The functor $\pi\taking I\to S$ sends objects 101,102,103 in $I$ to the object {\tt Employee}, in $S$; it similarly sends the arrow labeled worksIn in $\int(\delta)$ to the arrow labeled worksIn in $S$, etc. In the tables in (\ref{dia:instance on maincat}), which represents our instance $\delta$, there are 16 non-ID cells, whereas in $I=\int(\delta)$, there are only six arrows drawn. The other ten arrows have been left out of the picture of $I$ (e.g. the arrow $\LTO{102}\Too{\tn{Last}}\LTO{Russell}$ is not drawn) for readability reasons. The point is that the RDF triple store associated to instance $\delta$ is nicely represented using the standard Grothendieck construction. For example, the arrow $\LTO{101}\Too{\text{first}}\LTO{David}$ represents the RDF triple (101 :first David).
}
\label{fig:Grothendieck}
\end{figure}
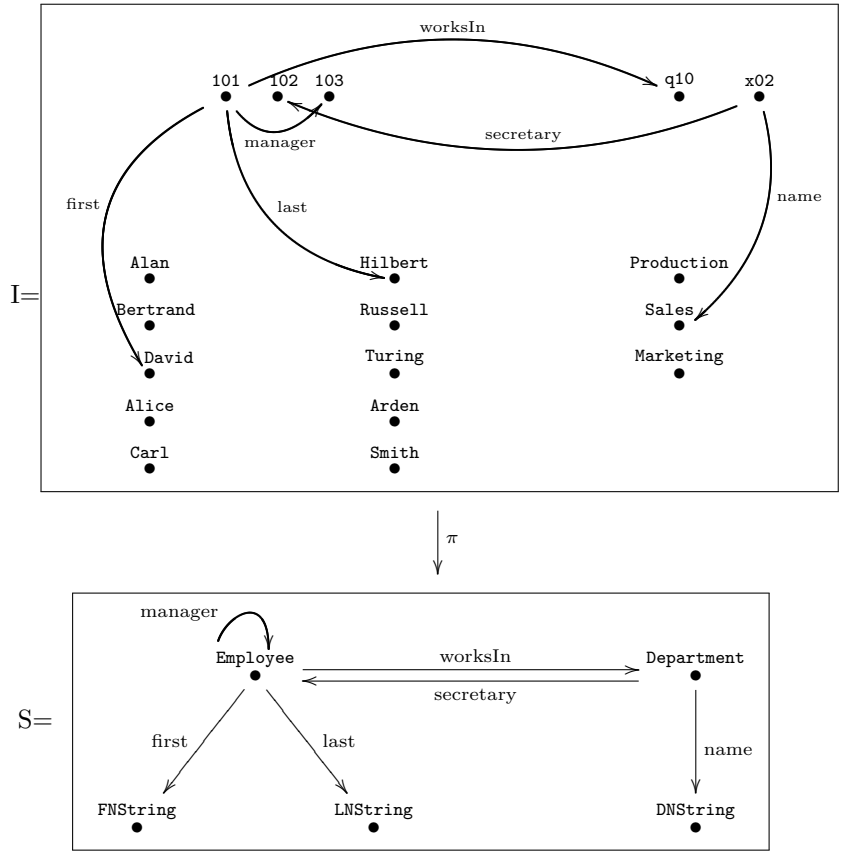

\end{example}

In the Introduction (Section \ref{sec:intro}), we discussed database instances in terms of mappings $\pi$, each from a data bundle $I$ to a base space $S$. We were referring to exactly the discrete opfibration picture in Figure \ref{fig:Grothendieck}. 

In Section \ref{sec:discrete ops} we will give a definition of discrete opfibrations in terms of lifting constraints (Definition \ref{def:dofib}). First, however, we attempt to understand a discrete opfibration $\pi\taking I\to S$ by considering its various fibers and their relationships. More precisely, given an object $s\in\Ob(S)$, we consider the fiber $\pi^\m1(s)$, and given a morphism $f\taking s\to s'$ in $S$ we consider how the fibers $\pi^\m1(s)$ and $\pi^\m1(s')$ relate. 

If $\pi\taking I\to S$ were not assumed to be a discrete opfibration but instead just a general functor, then all we would know about these various fibers would be that they are categories. But the first distinctive feature of a discrete opfibration is that the fiber $\pi^\m1(s)$ is a {\em discrete category}, i.e. a set, for each object $s\in S$; that is, there are no morphisms between different objects in a chosen fiber (see Proposition \ref{prop:discrete opfibrations are discrete}). The pre-image $\pi^\m1(f)$ of $f\taking s\to s'$ is a set of morphisms from objects in $\pi^\m1(s)$ to objects in $\pi^\m1(s')$. When $\pi$ is a discrete opfibration, there exists a unique morphism in $\pi^\m1(f)$ emanating from each object in $\pi^\m1(s)$, so the subcategory $\pi^\m1(f)\ss I$ can be cast as a single function $\pi^\m1(f)\taking\pi^\m1(s)\to\pi^\m1(s')$.

To recap, the discrete opfibration $\pi_\delta\taking\int(\delta)\to S$ of a set-valued functor $\delta\taking S\to\Set$ contains the same information as $\delta$ does, but a different perspective. We have $$\pi_\delta^\m1(s)\iso\delta(s) \hsp\tn{and}\hsp \pi_\delta^\m1(f)\iso\delta(f),$$ for any $s,s'\in\Ob(S)$ and $f\taking s\to s'$.

\setcounter{subsubsection}{\value{theorem}}
\subsubsection{Basic behavior of the Grothendieck construction}\addtocounter{theorem}{1}

Below are some simple results about the Grothendieck construction, all of which are well-known.

\begin{proposition}\label{prop:Grothendieck as pullback}

Let $\delta\taking S\to\Set$ be a functor. Then the Grothendieck construction $\int(\delta)\To{\pi_\delta} S$ of $\delta$ can be described as a pullback in the diagram of categories $$\xymatrix{\int(\delta)\ar[r]\ar[d]_{\pi_\delta}\ullimit&\Set_*\ar[d]^\pi\\S\ar[r]_\delta&\Set,}$$ where $\Set_*$ is the category of pointed sets and $\pi$ is the functor that sends a pointed set $(X,x\in X)$ to its underlying set $X$.

\end{proposition}

\begin{proof}

This follows directly from definitions.

\end{proof}

\begin{lemma}\label{lemma:morphisms on Grothendieck}

Let $S$ be a category. The functor $\int\taking S\set\to\Cat_{/S}$ is fully faithful. That is, given two instances, $\delta,\epsilon\taking S\to\Set$, there is a natural bijection, $$\Hom_{S\set}(\delta,\epsilon)\To{\iso}\Hom_{\Cat_{/S}}(\dispInt(\delta),\dispInt(\epsilon)).$$

\end{lemma}

\begin{proof}

This follows directly from definitions.
%

\end{proof}

\begin{proposition}\label{prop:pullback on Grothendieck}

Let $F\taking S\to T$ be a functor. Let $\delta\taking S\to\Set$ and $\epsilon\taking T\to\Set$ be instances, and suppose we have a commutative diagram \begin{align}\label{dia:pullback?}\xymatrix{\int(\delta)\ar[r]\ar[d]_{\pi_{\delta}}&\int(\epsilon)\ar[d]^{\pi_{\epsilon}}\\S\ar[r]_F&T.}\end{align} Then diagram (\ref{dia:pullback?}) is a pullback, i.e. $\int(\delta)\iso S\cross_T\int(\epsilon)$, if and only if $\delta\iso \pullb{F}\epsilon$.

\end{proposition}

\begin{proof}

This is checked easily by comparing the set of objects and the set of morphisms in $\int(\delta)$ with the respective sets in $S\cross_T\int(\epsilon)$.

\end{proof}

\setcounter{subsubsection}{\value{theorem}}
\subsubsection{Examples from algebraic topology}\label{sec:groupoids}\addtocounter{theorem}{1}
In algebraic topology (see \cite{May}), one associates to every topological space $X$ a fundamental groupoid $Gpd(X)$. It is a category whose objects are the points of $X$ and whose set of morphisms between two objects is the set of (equivalence classes of) continuous paths in $X$ from one point to the other. Two paths in $X$ are considered equivalent if one can be deformed to the other (without any part of it leaving $X$). Composition of morphisms is given by concatenation of paths.

One can reduce some of the study of a space $X$ to the study of this algebraic object $G=Gpd(X)$, and the latter is well-suited for translation to the language of this paper. 

\begin{example}

Suppose that $G$ is a groupoid. Then a covering of groupoids in the sense of \cite[Section 4.3]{May} is precisely the same as a surjective discrete opfibration with schema $G$.

Let $G=Gpd(S^1)$ denote the fundamental groupoid of the circle with circumference 1. Explicitly we have $\Ob(G)=\{\theta\in\RR\}/\!\!\sim$, where $\theta\sim\theta'$ if $\theta-\theta'\in\ZZ$; and we have $$\Hom_G(\theta,\theta')=\{x\in\RR\;|\;x+\theta\sim\theta'\}.$$ Think of $G$ as the category whose objects are positions of a clock hand and whose morphisms are arbitrary durations of time (rotating the hands from one clock position around and around to another). Consider the functor $T\taking G\to\Set$ such that $T(\theta)=\{t\in\RR\;|\;t-\theta\in\ZZ\}$ and such that for $x\in\Hom_G(\theta,\theta')$ we put $T(x)(t)=x+t$. So, for a clock position $\theta$, the functor $T$ returns all points in time at which the clock is in position $\theta$.

Applying the Grothendieck construction to $T$, we get a covering $\pi\taking\int(T)\to G$, which corresponds to the universal cover of the circle $S^1$. One can think of it as a helix  (modeling the time line) mapping down to the circle (modeling the clock). 

\end{example}

A much more sophisticated example relating databases to classical questions in algebraic topology may be found in \cite{Mor}.

%
%

\section{Constraints via lifting conditions}\label{sec:constraints}

In this section we introduce the lifting problem approach to database constraints. Roughly the same model will apply in the next section to database queries, the idea being that a lifting constraint is a lifting query that is guaranteed to have a result.

\subsection{Basic definitions}

\begin{definition}\label{def:lifting constraints}

Let $S\in\Cat$ be a database schema. A {\em (lifting) constraint on $S$} is a pair $(m,n)$ of functors $$W\To{m}R\To{n}S.$$ A functor $\pi\taking I\to S$ is said to {\em satisfy the constraint $(m,n)$} if, for all solid arrow commutative diagrams of the form \begin{align}\label{dia:lifting constraint}\xymatrix{W\ar[r]\ar[d]_m&I\ar[d]^{\pi}\\R\ar[r]_n\ar@{-->}[ur]&S,}\end{align} there exists a dotted arrow lift making the diagram commute.

A {\em (lifting) constraint set} is a set $\xi:=\{W_\alpha\To{m_\alpha}R_\alpha\To{n_\alpha}S\;|\;\alpha\in A\}$, for some set $A$. A functor $\pi\taking I\to S$ is said to {\em satisfy the constraint set $\xi$} if it satisfies each constraint $(m_\alpha,n_\alpha)$ in $\xi$.

Given a constraint set $\xi$ on $S$, we say that a constraint $W\To{m}R\To{n}S$ is {\em implied by $\xi$} if, whenever a functor $\pi\taking I\to S$ satisfies $\xi$ it also satisfies $(m,n)$.

\end{definition}

\begin{remark}

While not all constraints on databases are lifting constraints (for example, declaring a table to be the union of two others is not expressible by a lifting constraint), lifting constraints are the only type of constraint we will be considering in this paper. For that reason, we often leave off the word ``lifting," as suggested by the parentheses in Definition \ref{def:lifting constraints}.

\end{remark}

\begin{example}\label{ex:vertex is source}

Consider the schema 
$$\mcG=\fbox{\xymatrix{\LMO{E}\ar@/^1pc/[rr]^s\ar@/_1pc/[rr]_t&&\LMO{V}}}$$ 
The category $\mcG\set$ is precisely the category of (directed) graphs. Given a graph $X\taking\mcG\to\Set$, we have a function $X(s)\taking X(E)\to X(V)$ assigning to every edge its source vertex. Suppose we want to declare this function to be surjective, meaning that every vertex in $X$ is the source of some edge. We can do that with the following lifting constraint $$\fbox{$\LMO{V}$}\To{\hsps{m}}\fbox{\xymatrix{\LMO{E}\ar[r]^s&\LMO{V}}}\To{\hsps{n}}\fbox{\xymatrix{\LMO{E}\ar@/^1pc/[rr]^s\ar@/_1pc/[rr]_t&&\LMO{V}}}$$ where $m$ and $n$ respect labeling. A graph $\delta\taking\mcG\to\Set$ has the desired property, that every vertex is a source, iff $\int(\delta)$ satisfies the lifting constraint $(m,n)$.

\end{example}

\begin{definition}\label{def:universal constraint}

Let $S\in\Cat$ be a schema. Given a functor $m\taking W\to R$, define a set $\la m\ra$ of lifting constraints as follows: $$\la m\ra=\left\{W\To{m}R\To{n}S\;|\;n\in \Hom_\Cat(R,S)\right\}.$$ (Note that there is a bijection $\la m\ra\iso\Hom_\Cat(R,S)$, but the form of the set $\la m\ra$ allows us to apply Definition \ref{def:lifting constraints}.) Given a set of functors $M=\{m_j\taking W_j\to R_j\;|\;j\in J\}$, the union $$\la M\ra:=\bigcup_{j\in J}\la m_j\ra$$ is a constraint set, which we call the {\em universal constraint set generated by $M$}. A functor $\pi\taking I\to S$ satisfying the constraint set $\la M\ra$ is called an $M$-fibration. We say that elements of $M$ are {\em generating constraints} for $M$-fibrations.

\end{definition}

\begin{remark}

Universal constraint sets seem to be more important in traditional mathematical contexts than in ``informational" or database contexts. For example, in the world of simplicial sets, the Kan fibrations are $M$-fibrations for some universal constraint set $\la M\ra$, called {\em the set of generating acyclic cofibrations} (see \cite{Hir}).

\end{remark}

%
%

\subsection{Discrete opfibrations via lifting constraints}\label{sec:discrete ops}

Our goal now is to express the notion of discrete opfibrations in terms of lifting constraints. In other words, we will exhibit a finite set of functors $\{m_\alpha\taking W_\alpha\to R_\alpha\}_{\alpha\in A}$ that serve to ``check" whether an arbitrary functor $\pi\taking I\to S$ is a discrete opfibration. In fact, Definition \ref{def:dofib} will define $\pi$ to be a discrete opfibration if and only if it is a $\{\rho_1,\rho_2\}$-fibration, where $\rho_1\taking W_1\to R_1$ and $\rho_2\taking W_2\to R_2$ are functors displayed in Figure \ref{fig:constraints for discrete opfibration}.

\begin{figure}[h]
\caption{The generating constraints, $\rho_1$ and $\rho_2$, for discrete opfibrations} 
\label{fig:constraints for discrete opfibration}
$$W_1=\parbox{1.1in}{\fbox{\xymatrix@=19pt{&\parbox{.15in}{~\\\vspace{-.15in}}\\\LMO{a}&\hspace{.1in}&\\&\parbox{.15in}{~\\\vspace{-.15in}}}}}
\hspace{1.5in}
W_2=\parbox{1.1in}{\fbox{\xymatrix@=10pt{&\hsp&\LMO{b_1}\\\LMO{a}\ar[urr]^{f_1}\ar[drr]_{f_2}\\&&\LMO{b_2}}}}
$$
$$\hspace{-.6in}\xydoown{_{\parbox{.9in}{$\rho_1\taking W_1\to R_1$\\$\rho_1(a)=a$}}}
\hspace{1.6in}
\xydoown{_{\parbox{1.3in}{$\rho_2\taking W_2\to R_2$\\~\\$\rho_2(a)=a,$\\$\rho_2(b_1)=\rho_2(b_2)=b,$\\$\rho_2(f_1)=\rho_2(f_2)=f$}}}
$$
$$R_1=\parbox{1.1in}{\fbox{\xymatrix@=10pt{\LMO{a}\ar[rr]^f&\hsp&\LMO{b}}}}
\hspace{1.55in}
R_2=\parbox{1.1in}{\fbox{\xymatrix@=10pt{\LMO{a}\ar[rr]^f&\hsp&\LMO{b}}}}
$$
\end{figure}
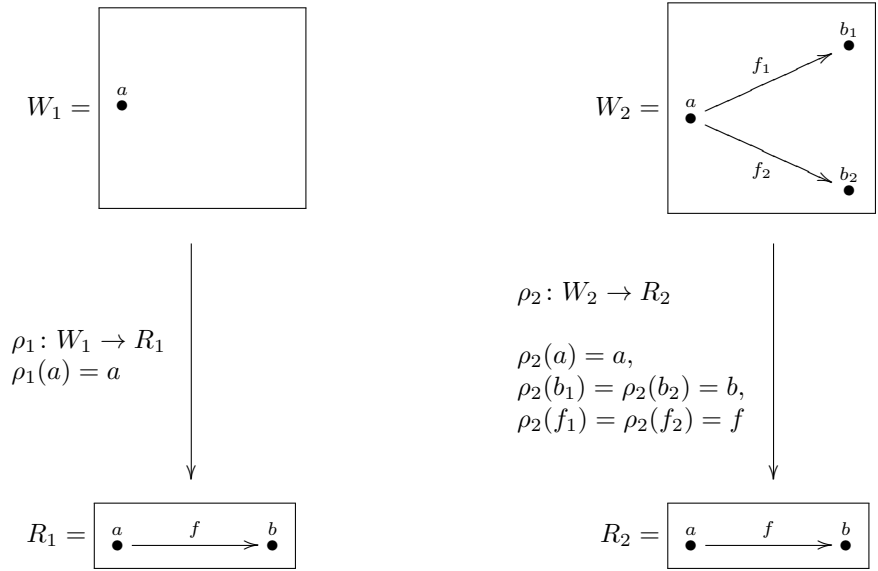

\begin{definition}\label{def:dofib}

Let $I$ and $S$ be categories and let $\pi\taking I\to S$ be a functor. Then $I$ is a {\em discrete opfibration} if it satisfies the lifting constraints $\rho_1$ and $\rho_2$ from Figure \ref{fig:constraints for discrete opfibration}. That is, for any pair of horizontal maps $W_1\to I$ and $R_1\to S$ (respectively for any pair of horizontal maps $W_2\to I$ and $R_2\to S$)
$$
\xymatrix@=30pt{
W_1\ar[r]^\forall\ar[d]_{\rho_1}&I\ar[d]^{\pi}\\
R_1\ar[r]_\forall\ar@{-->}[ur]^\exists&S}
\hspace{.9in}
\xymatrix@=30pt{
W_2\ar[r]^\forall\ar[d]_{\rho_2}&I\ar[d]^{\pi}\\
R_2\ar[r]^\forall\ar@{-->}[ur]^\exists&S}
$$ 
there exists a dotted arrow functor, as shown, such that the full diagram commutes.

\end{definition}

Let $\pi\taking I\to S$ be a $\{\rho_1,\rho_2\}$-fibration. Then for any functor $R_1=R_2\to S$, i.e. for any arrow $f\taking s\to s'$ in $S$, we have two lifting conditions. A good way to understand the conditions of Definition \ref{def:dofib} is that for any object $x\in\pi^\m1(s)$ in the fiber over $s$, 
\begin{enumerate}
\item there exists at least one arrow in $I$, emanating from $x$, whose image under $\pi$ is $f$, and
\item there exists at most one arrow in $I$, emanating from $x$, whose image under $\pi$ is $f$.
\end{enumerate}

In the remainder of this section we give some consequences of Definition \ref{def:dofib}. 

\begin{proposition}\label{prop:discrete opfibrations are discrete}

Let $\pi\taking I\to S$ be a discrete opfibration. Then for each object $s\in\Ob(S)$ the fiber $\pi^\m1(s)$ is a discrete category.

\end{proposition}

\begin{proof}

Let $s\in\Ob(S)$ be an object, and let $g\taking x\to y$ be a morphism in the fiber $\pi^\m1(s)\ss I$; we will show that $x=y$ and that $g=\id_x$ is the identity morphism. Consider the map $\rho_2\taking W_2\to R_2$ from Figure \ref{fig:constraints for discrete opfibration}. Let $n\taking R_2\to S$ be the functor sending $f$ to $\id_s$.  Let $p\taking W_2\to I$ send $f_1$ to $\id_x$ and send $f_2$ to $g$. We have a lifting diagram as in Definition \ref{def:dofib}, so a lift is guaranteed. This lift equates $\id_x$ and $g$.

\end{proof}


\begin{proposition}\label{prop:discrete opfibrations are faithful}

Let $\pi\taking I\to S$ be a discrete opfibration. Then $\pi$ is faithful. In other words, for any two objects $i,j\in\Ob(I)$ the function $$\pi\taking\Hom_I(i,j)\to\Hom_S(\pi(i),\pi(j))$$ is injective.

\end{proposition}

\begin{proof}

To prove that $\pi$ is faithful, we need only find a solution to each lifting diagram of the form:
\begin{align*}
W:=\fbox{\xymatrix{\LMO{i}\ar@/^.5pc/[r]^~\ar@/_.5pc/[r]_~&\LMO{j}}}
\xyright{}
I\\
\;\hspace{.3in}\xydown{_m}\hspace{.85in}\xydown{^\pi}~\hspace{-.15in}\\
R:=\fbox{\xymatrix{\LMO{i}\ar[r]&\LMO{j}}}
\xyright{}
S
\end{align*}
We can extend this diagram on the left with either surjective map from the relational constraint functor $\rho_2$ (see Figure \ref{fig:constraints for discrete opfibration}) to $m$, as indicated in the diagram:
$$\xymatrix{
W_2\ar[r]\ar[d]_{\rho_2}&W\ar[r]\ar[d]_m&I\ar[d]_\pi\\
R_2\ar@{=}[r]\ar@/_.5pc/@{-->}[urr]&R\ar[r]&S}$$
The result follows by noticing that the left-hand square is a pushout.

\end{proof}

Let $S$ be a category. We define a functor $\del\taking\Cat_{/S}\to S\set$ as follows. For any $F\taking X\to S$, let $\singlefun{X}\taking X\to\Set$ denote the terminal object of $X\set$ (see Notation \ref{sec:notation}), and note that $\int(\singlefun{X})\iso X$ in $\Cat_{/X}$. Define $\del(F)\taking S\to\Set$ as $$\del(F):=\lpush{F}(\singlefun{X}).$$ We have the following proposition, which is well-known.

\begin{proposition}\label{prop:adjunction}

\begin{enumerate}[(i)]
\item The functor $\del$ is left adjoint to $\int$: $$\Adjoint{\del}{\Cat_{/S}}{S\set.}{\dispInt}$$ 
\item For any $\gamma\taking S\to\Set$ the counit map is an isomorphism $$\del\circ\dispInt(\gamma)\To{\iso}\gamma.$$
\item An object $X\To{F} S$ in $\Cat_{/S}$ is a discrete opfibration if and only if $F\iso\int\del(F)$ in $\Cat_{/S}$.
\end{enumerate}

\end{proposition}

\begin{proof}

Let $F\taking X\to S$ be an object of $\Cat_{/S}$ and let $\gamma\taking S\to\Set$ be an object of $S\set$. By Proposition \ref{prop:pullback on Grothendieck} we have a pullback diagram: $$\xymatrix{\int(\pullb{F}\gamma)\ar[r]\ar[d]\ullimit&\int(\gamma)\ar[d]\\X\ar[r]_F&S}$$ which implies the first isomorphism in the following chain:
\begin{align*}
\Hom_{\Cat_{/S}}(F,\dispInt(\gamma))&\iso\Hom_{\Cat_{/X}}(\id_X,\dispInt(\pullb{F}\gamma))\\
&\iso\Hom_{X\set}(\singlefun{X},\pullb{F}\gamma)\\
&\iso\Hom_{S\set}(\lpush{F}(\singlefun{X}),\gamma)=\Hom_{S\set}(\del F,\gamma).
\end{align*} 
The second isomorphism follows from Lemma \ref{lemma:morphisms on Grothendieck} and the third is adjointness; this proves Statement (i). Statement (ii) follows from the same lemma.
 
By construction, $\pi\taking\int(\delta)\to S$ is a discrete opfibration for any $\delta\taking S\to\Set$, so if $X\To{F}S$ is not a discrete opfibration then $X\not\iso\int\del(F)$. Thus, it remains to show that if $F$ is a discrete opfibration then $X\iso\int\del(F)$. To see this, notice that for each $s\in\Ob(S)$ the set $F^\m1(s)$ is final in $(F\down s)$, so $$\del(F)(s)=\lpush{F}(\singlefun{X})(s)=\colim_{(F\down s)}\singlefun{X}\iso F^\m1(s).$$ This shows that the object structure in $F$ is the same as that in $\int\del(F)$. Similar analyses can be carried out for arrows and path equivalences.

\end{proof}

\subsection{Examples}\label{sec:examples}

In this section we will show how to use lifting constraints (see Definition \ref{def:lifting constraints}) to declare a number of different properties for tables in a database. Our examples include \begin{itemize}\item declaring a table to be non-empty,\item declaring a table to have exactly one row,\item declaring a foreign key to be injective,\item declaring a foreign key to be surjective, \item declaring a binary relation to be reflexive, symmetric, and/or transitive, \item declaring a table to be a product or a general limit of other tables, and \item declaring that there are no nontrivial cycles in the data on a self-referencing table.\end{itemize} We will discuss these in the above order.

\begin{example}[Nonempty]\label{ex:const nonempty}

Let $S$ be a schema, and let $T\in\Ob(S)$ be a table, which we want to declare non-empty. We use the constraint drawn as follows $$\parbox{.3in}{\fbox{\xymatrix{~}}}\hspace{-.1in}\To{\hsp m_1\hsp}\parbox{.35in}{\fbox{\xymatrix{\LMO{A}}}}\hspace{-.1in}\To{\hsp n\hsp}S$$ where $n(A)=T$. In other words, we set $W_1=\emptyset$ to be the empty category, and we set $R=\{A\}$ to be the discrete category with one object, $A$. To say that the lifting problem $$\xymatrix{W_1\ar[r]\ar[d]_{m_1}&I\ar[d]^\pi\\R\ar@{-->}[ur]\ar[r]_n&S}$$ has a solution is to say that there exists an object in the instance category $I$ whose image under $\pi$ is $T$. In other words, there exists a row in table $T$. Here, the commutativity of the upper-left triangle does nothing, and the commutativity of the lower-right triangle does all the work. 

\end{example}

\begin{example}[Cardinality=1]\label{ex:const card=1}

Let $S$ be a schema, and $T\in\Ob(S)$ a table, which we want to declare to have exactly one row. We know a constraint guaranteeing the existence of a row in $T$ from Example \ref{ex:const nonempty}; in Section \ref{sec:uniqueness} we will give a general method for transforming existence constraints into uniqueness constraints, but here we will just give the result of that method.

To declare $T$ to have at most one row, we use the constraint drawn as follows:  $$\parbox{.28in}{\fbox{\xymatrix{\LMO{a_1}\\\LMO{a_2}}}}\To{\hsp m_2\hsp}\parbox{.25in}{\fbox{\xymatrix{\LMO{A}}}}\To{\hsp n\hsp}S$$ where $m_2(a_1)=m_2(a_2)=A$ and where $n(A)=T$.  In other words, we set $W_2=\{a_1,a_2\}$ to be a discrete category with two objects, and we set $R=\{A\}$ to be a discrete category with one object. The lifting problem $$\xymatrix{W_2\ar[r]\ar[d]_{m_2}&I\ar[d]^\pi\\R\ar@{-->}[ur]\ar[r]_n&S}$$ has a solution iff both triangles commute. We know already that the image of $a$ and $b$ in $I$ consists of two rows in table $T$, because the square commutes. The commutativity of the upper-left triangle implies that $a$ and $b$ are the same, as desired. The commutativity of the lower-right triangle is implied by the surjectivity of $m_2$ and the commutativity of the square. 

The set $\{(m_1,n),(m_2,n)\}$ is a constraint set on $S$ that is satisfied by a discrete opfibration $\pi$ if and only if the set $I(T)$ of rows in $T$ has exactly one element.

\end{example}

We will be more brief from here on out. The following constraint was used in Example \ref{ex:vertex is source}.

\begin{example}[Surjective foreign key]\label{ex:const surj}

The declaration that a foreign key $f\taking T\to T'$ be surjective is achieved by the constraint: $$\parbox{.65in}{\fbox{\xymatrix{&\LMO{b}}}}\To{\hsp m\hsp}\parbox{.78in}{\fbox{\xymatrix{\LMO{A}\ar[r]^F&\LMO{B}}}}\To{\hsp n\hsp}S$$ where $m(b)=B, n(A)=T, n(B)=T'$, and $n(F)=f$.

\end{example}

\begin{example}[Injective foreign key]\label{ex:const inj}

The declaration that a foreign key $f\taking T\to T'$ be injective is achieved by the constraint: $$\parbox{.82in}{\fbox{\xymatrix@=10pt{\LMO{a_1}\ar[rrd]\\&&\LMO{b}\\\LMO{a_2}\ar[urr]}}}\To{\hsp m\hsp}\parbox{.8in}{\fbox{\xymatrix@=10pt{\LMO{A}\ar[rr]^F&&\LMO{B}}}}\To{\hsp n\hsp}S$$ where $m(a_1)=m(a_2)=A$ and $m(b)=B$, and where $n(F)=f$.

\end{example}

There exist constraints that ensure a binary relation $R\ss A\cross A$ is transitive, which we give in Example \ref{ex:transitive}. There is another constraint to ensure it is symmetric, and another to ensure it is reflexive; we leave these as exercises.

\begin{example}[Transitive binary relation]\label{ex:transitive}

The declaration that a relation 
$$\fbox{\xymatrix{R\ar@<.5ex>[r]^f\ar@<-.5ex>[r]_g& A}}\ss S
$$ 
be transitive is achieved by the constraint $$\parbox{1.6in}{~\\\\\\\fbox{\xymatrix@=8pt{&\LMO{r_1}\ar[ddl]_(.4){f_1}\ar[ddr]^(.4){g_1}&&\LMO{r_2}\ar[ddl]_(.4){f_2}\ar[ddr]^(.4){g_2}\\\\\LMO{a_1}&&\LMO{a_2}&&\LMO{a_3}}}\\~\vspace{.48in}}\parbox{.8in}{~\\\\\\$\Too{\hspace{.15in} m\hspace{.15in}}$\\~\vspace{.25in}}\parbox{1.7in}{\fbox{\xymatrix@=8pt{&\LMO{R_1}\ar[ddl]_(.4){F_1}\ar[ddr]^(.4){G_1}&&\LMO{R_2}\ar[ddl]_(.4){F_2}\ar[ddr]^(.4){G_2}\\\\\LMO{A_1}&&\LMO{A_2}&&\LMO{A_3}\\\\&&\LMO{R_3}\ar[uull]^{F_3}\ar[uurr]_{G_3}}}}\Too{\hspace{.15in} n\hspace{.15in}}S$$ where the functors $m$ and $n$ should be clear by our labeling (e.g. $n(R_1)=n(R_2)=n(R_3)=R$).

\end{example}

We now give the example of lifting constraints for products. This is part of a much larger story: In Section \ref{sec:lifting more expressive} we will show that any limit constraint can be modeled by lifting constraints.

\begin{example}[Product]\label{ex:const prod}

Suppose we have a table $T$ and two of its columns are $f\taking T\to U$ and $g\taking T\to V$. The declaration that (the set of rows in) table $T$ is the product of (the sets of rows in) tables $U$ and $V$ is achieved by two constraints, an existence constraint and a uniqueness constraint. The existence constraint is $$\parbox{.8in}{\fbox{\xymatrix@=10pt{\\\\\LMO{b}&&\LMO{c}}}}\To{\hsp m_1\hsp}\parbox{.9in}{\fbox{\xymatrix@=10pt{&\LMO{A}\ar[ddr]^G\ar[ddl]_F\\\\\LMO{B}&&\LMO{C}}}}\To{\hsp n\hsp}S$$ where $m_1(b)=B, m_1(c)=C$, and $n(F)=f, n(G)=g$. The uniqueness constraint is 
$$\parbox{1in}
{\fbox{
\xymatrix@=12pt{
&\LMO{a_1}\ar[ddr]^{G_1}\ar[ddl]_{F_1}\\&\LMO{a_2}\ar[dr]_(.6){G_2}\ar[dl]^(.6){F_2}\\\LMO{b}&&\LMO{c}}}}\To{\hsp m_2\hsp}\parbox{1in}{\fbox{\xymatrix@=12pt{&\LMO{A}\ar[ddr]^G\ar[ddl]_F\\\\\LMO{B}&&\LMO{C}}}}\To{\hsp n\hsp}S$$
 where $m_2(F_1)=m_2(F_2)=F, m_2(G_1)=m_2(G_2)=G,$ and $n(F)=f, n(G)=g$.

Thus the constraint set for $(T,f,g)$ to be a product is $\{(m_1,n),(m_2,n)\}$.

\end{example}

\begin{example}[Forests]

Let $S$ be the free category generated by the graph with one object and one arrow, pictured here: 
\begin{align}\label{dia:DDS}S:=\parbox{.5in}{\fbox{\xymatrix{\LMO{\nu}\ar@(l,d)[]+<0pt,-5pt>_p}}}\end{align}
This is just a self-referencing table. In mathematics, an instance $\delta\taking S\to\Set$ of such a self-referencing table is called a {\em discrete dynamical system} or {\em DDS}. The set $\delta(\nu)$ will be called the {\em set of nodes} of $\delta$ and given a node $x\in\delta(\nu)$, the node $\delta(p)(x)$ is called the {\em parent of $x$}. Here is a picture of a such an instance $\delta$ and its Grothendieck construction $I=\int(\delta)$.
\begin{align}\label{dia:DDS}\scriptsize
\delta:=\parbox{.6in}{
\begin{tabular}{| l || l |}\bhline
\multicolumn{2}{| c |}{{\tt $\nu$}}\\\bhline
{\bf ID}&{\bf p}\\\bbhline 
a&f\\\hline b&c\\\hline c&d\\\hline d&g\\\hline e&f\\\hline f&i\\\hline g&c\\\hline h&f\\\hline i&i\\\hline j&i\\\bhline
\end{tabular}
}
\normalsize\hspace{1in}
I:=\parbox{1.7in}{\fbox{\xymatrix@=8pt{
\LMO{a}\ar[drr]&&\LMO{b}\ar[rr]&&\LMO{c}\ar[rr]&&\LMO{d}\ar[d]\\
\LMO{e}\ar[rr]&&\LMO{f}\ar[drr]&&&&\LMO{g}\ar[llu]\\
\LMO{h}\ar[urr]&&&&\LMO{i}\ar@(r,u)[]+<0pt,10pt>\\
&&\LMO{j}\ar[urr]
}}}
\end{align}

Notice that a DDS looks like a forest (collection of trees) except that it may have cycles. These cycles can only occur at the root of a tree, and indeed each tree in the forest has a root cycle. In (\ref{dia:DDS}) we see that the tree containing $a$ has a root cycle of length 1, and the tree containing $b$ has a root cycle of length 3. Forests are a useful notion in computer science; we consider a DDS a forest if and only if each root cycle has length 1. This can be achieved by the following lifting constraint.

Let $R=S$ be the schema in (\ref{dia:DDS}), and let $n=\id\taking R\to S$. Let $W$ be the free category on the graph below, and let $m\taking W\to R$ denote the functor sending $p_1$ and $p_2$ to $p$.
$$\parbox{.85in}{\begin{center}W:=\end{center}\fbox{\xymatrix{\LMO{\nu_1}\ar@/^1pc/[r]^{p_1}&\LMO{\nu_2}\ar@/^1pc/[l]^{p_2}}}\begin{center}~\end{center}}
\Too{\hsp m\hsp}
\parbox{.5in}{\begin{center}R:=\end{center}\fbox{\xymatrix{\LMO{\nu}\ar@(l,d)[]+<0pt,-5pt>_p}}\begin{center}~\end{center}}
\Too{\hsp n\hsp}
\parbox{.5in}{\begin{center}S:=\end{center}\fbox{\xymatrix{\LMO{\nu}\ar@(l,d)[]+<0pt,-5pt>_p}}\begin{center}~\end{center}}
$$

\end{example}

\subsection{Encoding uniqueness constraints}\label{sec:uniqueness}

Suppose given a constraint $W\To{m}R\To{n}S$. According to Definition \ref{def:lifting constraints} a functor $\pi\taking I\to S$ satisfies $(m,n)$ if for every solid arrow diagram 
\begin{align}\label{dia:lifting constraint2}
\xymatrix{W\ar[r]\ar[d]_m&I\ar[d]^\pi\\R\ar[r]_n\ar@{-->}[ur]&S,}
\end{align} 
{\em there exists} a dotted arrow lift making it commute. Thus it may appear that all lifting constraints are existence declarations. However, by employing a technique found in \cite{Mak}, we can always turn such an existence declaration into a uniqueness declaration using a related lifting diagram. In fact this was done a couple times (see Examples \ref{ex:const card=1}, \ref{ex:const prod}) above. The uniqueness constraint corresponding to $(m,n)$ is 
\begin{align}\label{dia:uniqueness}
\xymatrix{R\amalg_WR\ar[rr]^-{(\id_R\;\amalg\;\id_R)}&&R\ar[r]^n&S.}
\end{align} 
In other words, $\pi$ satisfies constraint (\ref{dia:uniqueness}) if and only if there exists {\em at most one} dotted arrow lift making diagram (\ref{dia:lifting constraint2}) commute.

%

\subsection{Lifting constraints are more expressive than limit sketches}\label{sec:lifting more expressive}

In this section we show that lifting constraints are more expressive that limit sketches when it comes to set-models. We define limit sketches in Definition \ref{def:limit sketch}, prove that lifting constraints are at least as expressive as limit sketches in Proposition \ref{prop:lifting covers limit sketch}, and prove that lifting constraints are strictly more expressive than limit sketches in Proposition \ref{prop:lifting more expressive}.

For any category $\mcC$, we let $\mcC\lcone$ denote the category obtained by adjoining an initial object to $\mcC$.

\begin{definition}\label{def:limit sketch}

A {\em limit sketch} consists of a category $S$, and a set $D$ of commutative diagrams in $\Cat$ of the form 
$$\xymatrix{J_d\ar[r]^{X_d}\ar[d]_{i_d}&S\\(J_d)\lcone\ar[ru]_{L_d}}$$
one for each $d\in D$. Each $X_d$ is called a {\em specified limit pre-cone} in $S$ and each $L_d$ is called a {\em specified limit cone} in $S$. We call $S$ the {\em underlying category} of the sketch $(S,D)$.

If $(S,D)$ is a limit sketch, then an {\em $(S,D)$-model} is a functor $\delta\taking S\to\Set$ such that for each $d\in D$ the map $i_d$ induces an isomorphism
\begin{align}\label{dia:limit sketch bijection}
\lim_{(J_d)\lcone}(\delta\circ L_d)\iso\lim_{J_d}(\delta\circ X_d).
\end{align}

\end{definition}

\begin{proposition}\label{prop:lifting covers limit sketch}

Let $(S,D)$ be a limit sketch. Then we can construct a set of lifting constraints $\xi$ such that the functor $\int\taking S\set\to\Cat_{/S}$ induces a bijection between the set of functors $\delta\taking S\to\Set$ modeling $(S,D)$ and the set of instances $\pi\taking I\to S$ satisfying $\xi$.

\end{proposition}

\begin{proof}

It suffices to show that for each diagram $d=(J,X,L)$ as shown to the left
\begin{align}\label{dia:lift for limit}
\xymatrix{
&&&&\int(\delta)\ar[d]^\pi\\
J_d\ar[r]^{X_d}\ar[d]_{i_d}&S&\hsp&J_d\ar@{-->}[ur]\ar[r]^{X_d}\ar[d]_{i_d}&S\ar[r]^\delta&\Set\\
(J_d)\lcone\ar[ru]_{L_d}&&&J_d\lcone\ar[ru]_{L_d}}
\end{align}
there exists a set $K_d$ of lifting constraints $\{(m_k,n_k)\}_{k\in K_d}$ with the property that $\delta\taking S\to\Set$ satisfies (\ref{dia:limit sketch bijection}) if and only if $\int(\delta)\to S$ satisfies the constraints in $K_d$.

The limit $\lim_J\delta\circ X_d$ is in bijection with the set of dotted lifts $s_d\taking J_d\to\int(\delta)$ (such that $\pi\circ s_d=X_d$)  in the right-hand diagram of (\ref{dia:lift for limit}), and similarly the limit $\lim_((J_d)\lcone)\delta\circ L_d$ is in bijection with the set of lifts $(J_d)\lcone\to\int(\delta)$. Thus to say that $\delta$ models $d$ is to say that for every commutative diagram of the form
$$\xymatrix{
J_d\ar[r]^{s_d}\ar[d]_{i_d}&\int(\delta)\ar[d]^\pi\\(J_d)\lcone\ar[r]_{L_d}\ar@{-->}[ur]&S
}
$$
there exists a unique dotted lift. We thus take $K_d$ to be the set $\{(i_d,L_d),(i_d',L_d)\}$, where $i_d'\taking (J_d)\lcone\amalg_{J_d}(J_d)\lcone\to(J_d)\lcone$. In other words $(i_d,L_d)$ encodes the existence of the dotted lift and $(i'_d,L_d)$ encodes its uniqueness, as in Section \ref{sec:uniqueness}. This completes the proof.

\end{proof}

\begin{proposition}\label{prop:lifting more expressive}

There exists a schema $S$ and a set of lifting constraints $\xi$ on it whose satisfaction is not modeled by any limit sketch with underlying category $S$. 

\end{proposition}

\begin{proof}

Let $S=\ul{1}$ be the terminal category, so that a functor $\delta\taking S\to\Set$ can be considered as just a set $\delta\in\Ob(\Set)$ and we have $I:=\int(\delta)\iso\delta$. Consider the unique lifting constraint of the form $\ul{2}\To{m}\ul{1}\To{n}S$. Up to isomorphism there exists precisely two instance $\delta$ satisfying $\xi=\{(m,n)\}$, namely either $\delta\iso\ul{0}$ or $\delta\iso\ul{1}$. We will show that there is no limit sketch with underlying category $S$ having only two models up to isomorphism.

For a limit sketch on $S$ each $(J,X,L)$ either has $J=\emptyset$ or $J\To{X}S$ is an epimorphism. In the first case, $\lim_{(J_d)\lcone}(\delta\circ L_d)\iso\ul{1}$ and $\lim_{J_d}(\delta\circ X_d)\iso\delta$, so a model of $(J,X,L)$ must have $\delta\iso\ul{1}$. In the second case $\lim_{(J_d)\lcone}(\delta\circ L_d)\iso\delta\iso\lim_{J_d}(\delta\circ X_d)$, so every set $\delta$ models this constraint. Thus there is no set $D$ such that the set of sketch models of $(S,D)$ has precisely two elements up to isomorphism.

\end{proof}

\subsection{Constraint implications}\label{sec:constraint implications}

Propositions \ref{prop:retract implication} and \ref{prop:pushout implication} below are constraint implication results. That is, they show that instances satisfying one lifting constraint automatically satisfy another. These two constraint implications are not exhaustive, they merely give the idea. 

\begin{definition}

Suppose that one has a diagram of the form $$\xymatrix{W\ar[r]^{s_1}\ar[d]^{m} &W'\ar[r]^{p_1}\ar[d]^(.47){m'}&W\ar[d]^m\\R\ar[r]_{s_2}&R'\ar[r]_{p_2}&R}$$ such that the top and bottom compositions are identity, $$p_1\circ s_1=\id_W\hsp\tn{and}\hsp p_2\circ s_2=\id_R.$$ In this case we say that $m$ is a {\em retract} of $m'$. 

\end{definition}

\begin{proposition}\label{prop:retract implication}

Suppose that $(m,n)$ is a constraint for a schema $S$ and that $m$ is a retract of some $m'$, part of which is shown to the left in the diagram $$\xymatrix{W'\ar[r]^{p_1}\ar[d]_{m'} &W\ar[d]^{m}\\R'\ar[r]_{p_2}&R\ar[r]_n&S.}$$ Then any discrete opfibration $\pi\taking I\to S$ satisfying $(m',n\circ p_2)$ also satisfies $(m,n)$.

\end{proposition}

\begin{proof}

The proof is straightforward but we include it for pedagogical reasons. Suppose given a lifting problem 
\begin{align}\label{dia:lifting retracts}
\xymatrix{W\ar[r]^p\ar[d]_m&I\ar[d]^{\pi}\\R\ar@{-->}[ur]^\ell\ar[r]_{n}&S.}
\end{align}
We assume by hypothesis that the dotted arrow lift $f$ exists making the solid arrow diagram $$\xymatrix{W\ar[r]^{s_1}\ar[d]^(.6){m} &W'\ar[r]^{p_1}\ar[d]^(.6){m'}&W\ar[d]^(.6)m\ar[r]^p&I\ar[d]^{\pi}\\R\ar[r]_{s_2}&R'\ar@{-->}[urr]^(.4)f\ar[r]_{p_2}&R\ar[r]_n&S}$$ commute. But then one checks that $\ell=f\circ s_2\taking R\to I$ is a lift as in (\ref{dia:lifting retracts}).

\end{proof}

\begin{proposition}\label{prop:pushout implication}

Suppose that the square to the left in the diagram $$\xymatrix{W'\ar[r]\ar[d]_{m'}&W\ar[d]^{m}\\R'\ar[r]_q&R\lrlimit\ar[r]_n&S}$$ is a pushout (as indicated by the corner symbol $\ulcorner$). If $\pi\taking I\to S$ satisfies the constraint $(m',n\circ q)$ then it satisfies $(m,n)$.

\end{proposition}

\begin{proof}

Obvious.

\end{proof}

\section{Queries as lifting problems}\label{sec:queries as LPs}

In this section we will show a correspondence between queries and lifting problems, under which the set of results for a query corresponds to the set of solutions (i.e. lifts) for the associated lifting problem. The main example of this was discussed in Example \ref{ex:query party setup}. There we were interested in learning more about a married couple, given certain known information about them. After building up the necessary theory in Sections \ref{sec:where-less} and \ref{sec:general lifting} we will apply it to the case of the married couple in Example \ref{ex:bob and sue revisited}.

In the Introduction, more specifically in (\ref{dia:lifting SQL}), we alluded to a dictionary between certain SQL statements and lifting problems. In this section we will extend this a bit to include more specificity in the SELECT clause. Namely, we have this correspondence 
\begin{align}\label{dia:lifting SQL 2}
\parbox{.1in}{\xymatrix{&W\ar[r]^p\ar[d]_m&I\ar[d]^\pi\\X\ar[r]_q&R\ar@{-->}[ur]^\ell\ar[r]_n&S}}
\hspace{1in}
\parbox{1.4in}{
\begin{tabbing}
SELECT\;\; \=$X\To{q}R$\\
FROM \>$R\To{n}S$\\
WHERE \>$R\From{m}W\To{p}I$
\end{tabbing}}
\end{align} 
The map $q$ can be composed with any lift $\ell\taking R\to I$ to restrict our attention (i.e. project) to a certain segment of the result. We explain these ideas in Example \ref{ex:SQL statement}.  However, before getting to this general kind of query, we will discuss queries that do not include the WHERE-clause, i.e. the collection $W\to I$ of knowns.

\subsection{WHERE-less queries}\label{sec:where-less}

In this section we study queries as in Diagram (\ref{dia:lifting SQL 2}) in which the where-clause $W$ is empty, $W=\emptyset$. Such queries are often called {\em views}. In this case the two maps $R\From{m}W\To{p}I$ contain no information, so Diagram (\ref{dia:lifting SQL 2}) reduces to the following:
\begin{align*}
\parbox{.7in}{\xymatrix{&&I\ar[d]^\pi\\X\ar[r]_q&R\ar@{-->}[ur]^\ell\ar[r]_n&S}}
\hspace{1in}
\parbox{1.4in}{
\begin{tabbing}
SELECT\;\; \=$X\To{q}R$\\
FROM \>$R\To{n}S$
\end{tabbing}}
\end{align*}
We call these {\em WHERE-less queries}. 

\begin{definition}\label{def:probe}

Let $S$ be a schema. A {\em probe on $S$} is a functor $n\taking R\to S$; the category $R$ is called {\em the result schema for the probe}. Given a discrete opfibration $\pi\taking I\to S$ the probe $n$ is said to {\em set up the lifting problem} $$\xymatrix{&I\ar[d]^\pi\\R\ar[r]_n\ar@{-->}[ur]&S.}$$ In the presence of a discrete opfibration $\pi$, we may refer to the probe $n$ as {\em a where-less query}. We define the {\em set of solutions} to the query, denoted $\Gamma(n,\pi)$ as $$\Gamma(n,\pi):=\{\ell\taking R\to I\;|\;\pi\circ\ell=n\}.$$

\end{definition}

\begin{example}\label{ex:same last name}

Consider the discrete opfibration $\pi\taking I\to S$ given here: \begin{align*}I=\parbox{1.7in}{\fbox{\xymatrix@=7pt{&&&&&&\color{ForestGreen}{\LTO{Ann}}\\&&&&&&\color{ForestGreen}\LTO{Bob}\\\color{Blue}{\LMO{x137}}\ar[uurrrrrr]\ar[dddrrrrrr]&&&&&&\color{ForestGreen}{\LTO{Deb}}\\\color{Blue}{\LMO{x139}}\ar[uurrrrrr]\ar[ddrrrrrr]\\\color{Blue}{\LMO{x144}}\ar[rrrrrruu]\ar[ddrrrrrr]\\&&&&&&\color{red}{\LTO{Smith}}\\&&&&&&\color{red}{\LTO{Jones}}}}}\\\parbox{.9in}{$\xymatrix{~\ar[d]^\pi\\~}$}\\S=\parbox{1.7in}{\fbox{\xymatrix@=7pt{&&&&&\color{ForestGreen}{\LTO{FNames}}\\\color{Blue}{\LTO{Person}}\LA{urrrrr}{First}\LAL{drrrrr}{Last}\\&&&&&\color{red}{\LTO{LNames}}}}}\end{align*} To find two people with the same last name, we find lifts of the where-less query 
$$R:=\parbox{.95in}{
\fbox{
\xymatrix@=10pt{\color{Blue}{\LMO{P_1}}\ar[drr]^{L_1}\\&&\color{red}{\LMO{LN}}\\\color{Blue}{\LMO{P_2}}\ar[urr]_{L_2}}}}
\To{\hsp n\hsp}
\parbox{1.7in}{
\fbox{
\xymatrix@=7pt{&&&&&\color{ForestGreen}{\LTO{FNames}}\\\color{Blue}{\LTO{Person}}\LA{urrrrr}{First}\LAL{drrrrr}{Last}\\&&&&&\color{red}{\LTO{LNames}}}}}=S
$$ 
where both $n(L_1)=n(L_2)=(\LTO{Person}\To{\tn{Last}}\LTO{LNames}$). There are two people (Ann Smith, Bob Smith) with the same last name, so we may hope to get as our result set $\{(x137,\text{Smith},x139)\}$.

Here is how to compute the result set for our query. We are looking for functors $\ell\taking R\to I$ that make the diagram 
\begin{align}\label{dia:same last name}
\xymatrix{&I\ar[d]^\pi\\R\ar[r]_n\ar[ur]^\ell&S}
\end{align} 
commute. Since $L_1$ and $L_2$ in $R$ are sent to Last in $S$, we need to choose two ``downward sloping" arrows in $I$ with the same target. Doing so, we indeed find all pairs of persons in $I$ that have the same last name. Unfortunately, this query would return five results, which we can abbreviate as 
\begin{align}\label{dia:query results}
&(x137,\text{Smith},x139),\hsp (x139,\text{Smith},x137),\\
\nonumber &(x137,\text{Smith},x137),\hsp(x139,\text{Smith},x139),\hsp(x144,\text{Jones},x144).
\end{align} 
The first two are what we are looking for, but they are redundant; the last three are degenerate (e.g. Deb Jones has the same last name as Deb Jones). We will deal with these issues in Example \ref{ex:dedupe the query}, after we discuss morphisms of queries. 

\end{example}

\begin{definition}\label{def:strict morphisms}

Let $S$ be a schema. Given two probes $n_1\taking R_1\to S$ and $n_2\taking R_2\to S$, we define a {\em strict morphism} from $n_1$ to $n_2$, denoted $f\taking n_1\to n_2$, to be a functor $f\taking R_1\to R_2$ such that $n_2\circ f=n_1$.  Let $\Prbs(S)\iso\Cat_{/S}$ denote the category whose objects are probes and whose morphisms are strict morphisms. In the presence of a discrete opfibration $\pi\taking I\to S$, we may refer to $f$ as a {\em strict morphism of where-less queries} (as in Definition \ref{def:probe}).

Given a strict morphism $f\taking n_1\to n_2$, one obtains a function $\Gamma(f,\pi)\taking\Gamma(n_2,\pi)\to\Gamma(n_1,\pi)$, because any lift $\ell_2$ in the diagram \begin{align}\label{dia:strict}\xymatrix{&&I\ar[d]^\pi\\R_1\ar[r]^f\ar@/_1pc/[rr]_{n_1}&R_2\ar[r]^{n_2}\ar[ur]^{\ell_2}&S,}\end{align} i.e. with $n_2=\pi\circ\ell_2$, induces a lift $\ell_1:=\ell_2\circ f\taking R_1\to I$ with $n_1=\pi\circ\ell_1.$ We thus have produced a functor $\Gamma(-,\pi)\taking\Prbs(S)\op\to\Set$. It is just the representable functor at $\pi$, 
$$\Gamma(-,\pi)=\Hom_{\Prbs(S)}(-,\pi).$$

\end{definition}

\begin{remark}

We use the term {\em strict} morphism of probes in Definition \ref{def:strict morphisms} because a more lax version of morphism will be defined later, in Definition \ref{def:probes}. Whereas above we consider commutative triangles of categories (e.g. $n_2\circ f=n_1$ in (\ref{dia:strict})) and call the resulting category $\Prbs(S)$, the lax version will allow for natural transformations (e.g. $n_2\circ f\Rightarrow n_1$) and will be denoted $\Prb(S)$. The functor $\Gamma(-,\pi)\taking\Prbs(S)\to\Set$ defined in Definition \ref{def:strict morphisms} can be extended to a functor (which we give the same name), $\Gamma(-,\pi)\taking\Prb(S)\to\Set$. This will all be discussed in Section \ref{sec:new discrete opfibration}.

\end{remark}

\begin{example}\label{ex:dedupe the query}

We again consider the situation from Example \ref{ex:same last name}, where we were using the query $n\taking R\to S$ to look for pairs of people who had the same last name. The solution set in (\ref{dia:query results}) had two problems: \begin{itemize}\item we were getting degenerate answers because every person has the same last name as him- or her-self, and \item we were getting order-redundancy because, given two people with the same last name, we can reverse the order and get another such pair.\end{itemize}

In order to deal with the first issue, consider the strict morphism $f$ of queries $$\parbox{.9in}{\begin{center}$R=$\end{center}\fbox{\xymatrix@=10pt{\color{Blue}{\LMO{P_1}}\ar[drr]^{L_1}\\&&\color{red}{\LMO{LN}}\\\color{Blue}{\LMO{P_2}}\ar[urr]_{L_2}}}}\xyright{^f}\parbox{.9in}{\vspace{-.2in}\begin{center}$R_2:=$\end{center}\fbox{\xymatrix@=10pt{\color{Blue}{\LMO{P}}\ar[rr]^{L}&&\color{red}{\LMO{LN}}}}}\xyright{^{n_2}}\parbox{1.5in}{\begin{center}$S=$\end{center}\fbox{\xymatrix@=4pt{&&&&&\color{ForestGreen}{\LTO{FNames}}\\\color{Blue}{\LTO{Person}}\LA{urrrrr}{First}\LAL{drrrrr}{Last}\\&&&&&\color{red}{\LTO{LNames}}}}}$$ where $f(L_1)=f(L_2)=L$, and note that indeed $n=n_2\circ f$. By Definition \ref{def:strict morphisms} this induces a function between the solution sets; i.e. we get a function $$\Gamma(f,\pi)\taking\Gamma(n_2,\pi)\to\Gamma(n,\pi).$$ In our example (\ref{dia:query results}), the image of this function is precisely the set of duplicates. In other words, if we delete the elements in the image of $\Gamma(f,\pi)$ we get $$\Gamma(n,\pi) - \Gamma(n_2,\pi) = \{(x137, \text{Smith}, x139), (x139, \text{Smith}, x137)\}.$$

In order to deal with the remaining order-redundancy issue, consider the swap map $s\taking R\to R$ given by $s(L_1)= L_2$ and $s(L_2)= L_1$. Note that $n\circ s=n$. Thus we have a strict morphism of probes $s\taking n\to n$, which induces a function $\Gamma(s,\pi)\taking\Gamma(n,\pi)\to\Gamma(n,\pi)$. By taking the orbits of this function, we effectively quotient out by order-swapping. In fact our swap map acts not just on $(R,n)$ but on $(R_2,n_2)$ as well, and so we can combine this method with the one above to obtain the desired answer, the one element set consisting of $(x137, \text{Smith}, x139)$, in unspecified order.

\end{example}

\begin{proposition}\label{prop:where-less}

Let $\delta\taking S\to\Set$ be an instance and $\pi_{\delta}\taking I\to S$ the induced discrete opfibration. Given any probe $n\taking R\to S$, there is an isomorphism $$\Gamma(n,\pi_{\delta})\To{\iso}\tn{lim}_R (\delta\circ n).$$

\end{proposition}

\begin{proof}

Consider the diagram $$\xymatrix{&I\ar[r]\ar[d]_{\pi_\delta}\ullimit&\Set_*\ar[d]^\pi\\R\ar[r]_n&S\ar[r]_\delta&\Set}$$ where the right-hand square is a pullback, as shown in Proposition \ref{prop:Grothendieck as pullback}. We have a bijection $$\Hom_{\Cat_{/S}}(n,\pi_\delta)\iso\Hom_{\Cat_{/\Set}}(\delta\circ n,\pi).$$ The left-hand side is $\Gamma(n,\pi_\delta)$ and the right-hand side is a standard formula for the limit of a set-valued functor, in this case for $\lim_R(\delta\circ n)$.

%

\end{proof}

\subsection{Binding variables}\label{sec:binding variables}

In Section \ref{def:lifting constraints} we defined lifting constraints on a schema $S$ to be a pair of composable functors $W\To{m}R\To{n}S$. The idea is to think of $R$ as a set of equations (or a {\em join graph} and of $W$ as a set of variables to be bound at run-time. Our lifting approach below for queries will assume that the variables (in $W$) have already been bound to something in the active domain of $\pi$. As mentioned in the introduction, it is not standard to allow queries to depend on instances. In this short section we explain how to use where-less queries to determine the active domains. In this way, we will explain how instance-independent queries can be posed using the same lifting-problems approach. 

The idea is reminiscent of what is known in modern database practice as a {\em cursor}. Once the active domains for the variables in $W$ are found, one can either run the cursor (i.e. the query) parameterized over all values in these active domains, or prompt the user to choose bindings for these variables.

We assume for this section that $W$ is a discrete category; this will be most common in practice, but regardless all the ideas we will now discuss generalize to the non-discrete case.

Suppose given a cursor $W\To{m}R\To{n}S$. To determine the active domains of each variable in $W$, we simply apply the where-less query given by the diagram 
$$
\xymatrix{
&I\ar[d]^\pi\\
W\ar@{-->}[ur]\ar[r]_{n\circ m}&S
}
$$
The set of lifts $\Gamma(n\circ m,\pi)$ is the set of possible variable bindings. Once a lift $p\taking W\to I$ is chosen, we have a commutative square
$$\xymatrix{W\ar[r]^p\ar[d]_m&I\ar[d]^\pi\\R\ar@{-->}[ur]^\ell\ar[r]_n&S}
$$ 
and as we will see in Section \ref{sec:general lifting} below, the dotted arrow lifts $\ell$ will correspond to the results of the now-fully-defined query.

There is one more case we should discuss. Suppose one wants to pose a query such that it is not known in advance whether the chosen constants will or will not be available in the active domain---if they are not, the query must certainly return an empty set of results, and this is the intended behavior. In fact, this is the type of situation that is most often called a query in database literature. In the remainder of Section \ref{sec:binding variables}, we explain how this is handled by lifting queries.

Let $Dom$ denote the set of all possible domain values, let $\ol{Dom}$ denote the indiscrete category on $Dom$, and let $d\taking \ol{Dom}\times S\to S$ denote the projection.
\footnote{If one wants each table $s\in\Ob(S)$ to have its own data type, replace $d$ with the appropriate category over $S$.} 
Recall that for any category $\mcC$, the set of functions $\Ob(\mcC)\to Dom$ is in natural bijection with the set of functors $I\to\ol{Dom}$. 

We are given the shape of the query $W\To{m}R\To{n}S$, and we are also given, for each $w\in W$ a value $t(w)\in Dom$. In other words, our query is represented by the commutative square to the left
$$
\parbox{1.2in}{
\xymatrix@=35pt{
W\ar[r]^-{(t,n\circ m)}\ar[d]_m&\ol{Dom}\times S\ar[d]^d\\
R\ar[r]_n&S
}}
\hspace{.75in}
\parbox{1.2in}{
\xymatrix@=25pt{
&I\ar[d]^{(v,\pi)}\\
W\ar@{-->}[ur]^p\ar[r]\ar[d]_m&\ol{Dom}\times S\ar[d]^d\\
R\ar@{-->}[uur]\ar[r]_n&S
}}
$$
We also have are given a map $v\taking I\to\ol{Dom}$ that sends each datum $i\in\Ob(I)$ to its value in $Dom$. Form the solid-arrow diagram as to the right. As above, we perform the query in two steps. First we find all lifts $p\taking W\to I$ such that $(v,\pi)\circ p=(t,n\circ m)$. If this set is empty then the query will return an empty result set. However, if there do exist lifts $p$, then by choosing one, we bind our $W$-variables to their values found in the active domain. Finally, for each one we find all lifts $R\to I$ making the diagram to the right above commute. The set of all ways to do this is the set of results for our query.

\subsection{General lifting queries}\label{sec:general lifting}

In this section we tackle the more general lifting query. These closely resemble graph pattern queries, as used in SPARQL (see \cite{PS}). We will show how to perform queries like (and including) the one suggested in Example \ref{ex:query party setup}, where we hoped to find the last names of our new acquaintances, Bob and Sue. We begin with the definition.

\begin{definition}\label{def:query}

Let $S$ be a schema and $\pi\taking I\to S$ a discrete opfibration. A {\em query on $\pi$} is a solid-arrow commutative diagram of the form 
\begin{align}\label{dia:define queries}
\xymatrix{W\ar[r]^p\ar[d]_m&I\ar[d]^\pi\\R\ar@{-->}[ur]^\ell\ar[r]_n&S}
\end{align}
The categories $W$ and $R$ are called the {\em where-category} and the {\em result schema}, respectively. We define the {\em set of solutions} to the query, denoted $\Gamma^{m,p}(n,\pi)$, to be the set of lifts $\ell$ making the diagram commute. Precisely, $$\Gamma^{m,p}(n,\pi):=\{\ell\taking R\to I\;|\; \pi\circ\ell=n\;\;\tn{and}\;\;\ell\circ m=p\}.$$ 

\end{definition}

\begin{example}\label{ex:SQL statement}

By this point, we have developed the theory necessary to make sense of the following dictionary.

\begin{align*}
\parbox{.1in}{\xymatrix{&W\ar[r]^p\ar[d]_m&I\ar[d]^\pi\\X\ar[r]_q&R\ar@{-->}[ur]^\ell\ar[r]_n&S}}
\hspace{1in}
\parbox{1.4in}{
\begin{tabbing}
SELECT\;\; \=$X\To{q}R$\\
FROM \>$R\To{n}S$\\
WHERE \>$R\From{m}W\To{p}I$
\end{tabbing}}
\end{align*} 

Each lift $\ell$ in the commutative square is a solution to the SELECT $\ast$ statement, and composing $\ell$ with $q$ projects to schema $X$.

\end{example}

The following proposition says that for any query on a dataset $\delta$, there is a canonical embedding of the query result back into $\delta$.

\begin{proposition}\label{prop:query result state}

Let $\delta\taking S\to\Set$ be a instance on a schema and $\pi\taking I\to S$ the associated discrete opfibration. Suppose given a query (lifting problem) $$\xymatrix{W\ar[r]^p\ar[d]_m&I\ar[d]^{\pi}\\R\ar[r]_n\ar@{-->}[ur]&S}$$ with solution set $\Gamma^{m,p}(n,\pi)\in\Set.$ Considering this set as a constant functor $\Gamma\taking R\to\Set$ (given by $\Gamma(r)=\Gamma^{m,p}(n,\pi)$ for all $r\in\Ob(R)$), there is an induced map of $R$-sets, $$\Res\taking\Gamma\to \pullb{n}\delta.$$

\end{proposition}

\begin{proof}

Let $\Gamma(n,\pi)=\{\ell\taking R\to I \;|\;\pi\circ\ell=n\}$ denote the set of solutions to the where-less query $n\taking R\to S$. Clearly, we have an inclusion $\Gamma^{m,p}(n,\pi)\inj\Gamma(n,\pi)$. By Proposition \ref{prop:where-less}, there is an isomorphism $\Gamma(n,\pi)\iso\lim_R(\delta\circ n).$ 

Let $t\taking R\to\ul{1}$ denote the terminal functor. It follows from definitions that for any functor $G\taking R\to\Set$, there is an isomorphism of $[0$]-Sets, $\lim_{R}(G)\iso \rpush{t}(G)$, so in particular we have an inclusion $\Gamma^{m,p}(n,\pi)\to \rpush{t}(\delta\circ n).$ By the $(\pullb{t},\rpush{t})$-adjunction, there is an induced map $$\pullb{t}(\Gamma^{m,p}(n,\pi))\to(\delta\circ n)$$ of $R$-sets. The result follows, since $\pullb{t}(\Gamma^{m,p}(n,\pi))=\Gamma$ and $\delta\circ n=\pullb{n}\delta$. 

\end{proof}

\begin{example}[Bob and Sue, revisited]\label{ex:bob and sue revisited}

The motivating example for this paper was presented in Section \ref{sec:main example}. In particular, we provided a SPARQL query to find all instances of married couples with the requisite characteristics (e.g. the husband's and wife's first names being Bob and Sue respectively). We showed that this SPARQL query could be straightforwardly transformed into a lifting problem of the form 
$$\xymatrix{
&W\ar[r]^p\ar[d]_m&I\ar[d]^\pi\\
X\ar[r]_q&R\ar[r]_n\ar@{-->}[ur]^\ell&S
}
$$
as in (\ref{dia:qp2 lifting}), and we specified the two functors $W\To{m}R\To{n}S$. We did not specify the discrete opfibration $I\To{\pi}S$ or the inclusion of the known data $p\taking W\to I$, because writing out a convincing possibility for $I$ would necessitate too much space to be worthwhile in this document.

The lifting diagram (\ref{dia:qp2 lifting}) was presumed to have only one solution, because it was presumed that we knew enough about Bob and Sue that no one else fit the description. In the language of Definition \ref{def:query}, the set $\Gamma^{m,p}(n,q)$ has one element. By Proposition \ref{prop:query result state}, this element can be written as a database state on $R$. We output the result as a two-level table with one row in (\ref{dia:result state}), repeated here, 
\begin{align*}
\footnotesize
\begin{tabular}{| c ||| c || c | c | c | c || c | c | c |}
\bhline
\multicolumn{9}{| c |}{\bf Marriage}\\
\bhline 
\multirow{2}{*}{\bf ID}&\multicolumn{4}{ c |}{\bf Husband}&\multicolumn{4}{ c |}{\bf Wife}\\\cline{2-9}\cline{2-9}
&{\bf ID}&{\bf First}&{\bf Last}&{\bf City}&{\bf ID}&{\bf First}&{\bf Last}&{\bf City}\\
\bbbhline G3801&M881-36&Bob&Graf&Cambridge&W913-55&Sue&Graf&Cambridge\\
\bhline
\end{tabular}
\end{align*}
which in fact was a state on a schema $X\To{q} R$, where $X$ is the schema
\begin{align}
\tiny\parbox{4.3in}{
\begin{center}X:=\end{center}
\fbox{\xymatrix@=14pt{&&&\obox{G}{.4in}{Marriage}\ar[dll]\ar[drr]\\&\obox{P1}{.4in}{Husband}\ar[d]\ar[dr]\ar[dl]&&&&\obox{P2}{.2in}{Wife}\ar[d]\ar[dl]\ar[dr]&\\\obox{F1}{.2in}{First}&\obox{L1}{.2in}{Last}&\obox{C1}{.2in}{City}&&\obox{C2}{.2in}{City}&\obox{L2}{.2in}{Last}&\obox{F2}{.2in}{First}}}
}
\end{align}

While we have not discussed two-level tables before, we hope the idea is straightforward.

\end{example}

\subsection{SPARQL queries involving predicate variables}

In Example \ref{ex:query party setup} our SPARQL query (\ref{dia:SPARQL}) only has variables in subject and object positions (the nodes of the schema). It seems that most SPARQL queries used in practice also only have variables in the subject and object positions (see, e.g. \cite{DZS}); still, general SPARQL queries can involve variables in any position including in predicate positions, which correspond to the arrows of the schema. For example, we may use 
\begin{align}\label{dia:John Mary}
\tn{(John ?x Mary)}
\end{align}
to find all known relationships between John and Mary. To deal with this type of query, one may proceed as follows. 

If $S=(V,E,s,t)$ is a graph (thought of as a schema with trivial path equivalences, which is in keeping with RDF schemas), then $S$ itself can be viewed as a database instance $S\taking\mcG\to\Set$ on the schema 
$$\mcG=\fbox{\xymatrix{\LTO{Verb}\ar@/^1pc/[rr]^{\tn{subj}}\ar@/_1pc/[rr]_{\tn{obj}}&&\LTO{Noun}}}$$ 
similar to Example \ref{ex:vertex is source}. We will be working with the Grothendieck construction $\int(S)\to\mcG$. The category $\int(S)$ will be generated by a bipartite graph. The set of vertices in $\int(S)$ is the union $\tt{Noun}\amalg \tt{Verb}$, we might call them verb vertices and noun vertices. There is a unique edge in $\int(S)$ from every verb vertex to its subject noun vertex and another to its object noun vertex. 

\begin{example}\label{ex:elements of graph}

Let $X=(V,E,s,t)$ be the graph to the left below:
$$X:=\parbox{1.5in}{\fbox{
\xymatrix{
\LTO{John}\LA{r}{livesIn}\LAL{d}{sonOf}&\LTO{Iowa}\\\LTO{Mary}
}}}\hspace{.5in}
\varint(X):=\parbox{1.5in}{
\fbox{\xymatrix{
\LTO{John}&\LTO{livesIn}\LAL{l}{subj}\LA{r}{obj}&\LTO{Iowa}\\
\LTO{sonOf}\LAL{u}{subj}\LA{d}{obj}\\
\LTO{Mary}
}}}
$$
If $X$ is conceived as an instance $X\taking\mcG\to\Set$, then $\int(X)$ is the category to the right above. 

\end{example}

An instance $\pi\taking I\to S$ can be considered simply as a map of graphs, i.e. a map of instances on $\mcG$. Taking its Grothendieck construction yields a functor $\int(I)\to\int(S)$, whereby each arrow from $S$ (representing a foreign key column) and each arrow from $I$ (representing a cell in a foreign key column) have become a vertex in $\int(S)$ and $\int(I)$ respectively, as in Example \ref{ex:elements of graph}. We can perform the original SPARQL query (\ref{dia:John Mary}) to this derived form of the database because our original predicate can now be accessed as a subject or an object. For example our statement (John ?x Mary) would become the pair of statements (?x subj John) (?x obj Mary).

\section{The category of queries on a database}\label{sec:formal properties}

In this section we will discuss some formal properties of the machinery developed in earlier sections. For example we will show that the queries on a given database can be arranged into a database of their own and subsequently queried. This process is commonly known as nesting queries. To this end, we define a category of queries and prove that the process of finding solutions is functorial. We do this in Sections \ref{sec:new discrete opfibration} and \ref{sec:category of queries}. In Section \ref{sec:migration} we extend some results from Section \ref{sec:constraint implications}, giving more detail on the interaction between data migration functors, on the one hand, and query containment and constraint implication on the other.

This section is technical, but it may have fruitful applications. Given any database $\pi$, the category $\Qry(\pi)$ organizes the queries (or views) on $\pi$ into a schema of their own. There is a canonical instance on $\Qry(\pi)$ populating each table (corresponding to a query) with its set of results. In typical applications, users of a database $\pi$ are often better served by interacting with $\Qry(\pi)$ rather than with $\pi$. It is important to understand how schema evolution affects different parts of $\Qry(\pi)$; this is briefly discussed in Section \ref{sec:migration}.

\subsection{New discrete opfibrations from old}\label{sec:new discrete opfibration}

The following theorem is not new, but its formulation in terms of databases is. Furthermore, the proof may be instructive.

\begin{theorem}\label{thm:query discrete opfibration}

Let $\pi\taking I\to S$ be a discrete opfibration and let $B$ be a category. Then the induced functor $\pi^B\taking I^B\to S^B$ is a discrete opfibration. If $\delta^B=\del(\pi^B)\taking S^B\to\Set$ is the associated instance, then for any $F\taking B\to S$ in $\Ob(S^B)$, there is a bijection $$\delta^B(F)\iso\Gamma(F,\pi).$$

\end{theorem}

\begin{proof}

We begin our proof of the first claim by drawing a figure for reference: 
\begin{align}\label{dia:natural transformation}
\xymatrix{&&I\ar[dd]^\pi\\\\
B\ar[uurr]^{\ell_1}\ar@/^1pc/[rr]^{F_1}\ar@/_1pc/[rr]_{F_2}\ar@{}[rr]|{\Down\alpha}&&S}
\end{align} 
To see that $\pi^B$ is a discrete opfibration, suppose that $F_1,F_2\taking B\to S$ are functors and $\alpha\taking F_1\to F_2$ is a natural transformation. Given a functor $\ell_1\taking B\to I$ with $\pi\circ\ell_1=F_1$, we must show that there exists a unique functor $\ell_2\taking B\to I$ and natural transformation $\beta\taking\ell_1\to\ell_2$ such that $\pi\circ\ell_2=F_2$ and $\pi\circ\beta=\alpha$. For any object $b\in\Ob(B)$, the map $\alpha_b\taking F_1(b)\to F_2(b)$ in $S$ together with the object $\ell_1(b)\in I$, such that $\pi(\ell_1(b))=F_1(b)$, induces a unique arrow $\beta_b\taking\ell_1(b)\to i_b$ in $I$ for some $i_b\in\Ob(I)$, because $\pi$ is a discrete opfibration. Define $\ell_2(b)=i_b$. This defines $\ell_2\taking B\to I$ on objects. 

Now suppose that $f\taking b\to b'$ is any morphism in $B$. Applying what we have so far, we get a functor $X\to Y$, where $X$ is the solid-arrow portion of the category to the left and $Y$ is the commutative square category to the right, 
\begin{align}\label{dia:crescent square}
X:=\parbox{1.4in}{\fbox{\xymatrix{\ell_1(b)\ar[r]^{\beta_b}\ar[d]_{\ell_1(f)}&\ell_2(b)\ar@{..>}[d]^?\\\ell_1(b')\ar[r]_{\beta_{b'}}&\ell_2(b')}}} 
\To{\hsp}
\parbox{1.6in}{\fbox{\xymatrix{F_1(b)\ar[r]^{\alpha_b}\ar[d]_{F_1(f)}&F_2(b)\ar[d]^{F_2(f)}\\F_1(b')\ar[r]_{\alpha_{b'}}&F_2(b')}}}=:Y
\end{align}
and we get a commutative diagram $$\xymatrix{X\ar[r]\ar[d]&I\ar[d]^\pi\\Y\ar[r]&S}$$ In order to complete our definition of $\ell_2$, our goal is to fill in the missing side (the dotted arrow labeled ``?") in square $X$. 

The map $F_2(f)\taking F_2(b)\to F_2(b')$ in $S$ together with the object $\ell_2(b)\in\Ob(I)$ with $\pi(\ell_2(b))=F_2(b)$ induces a unique arrow $h_{b'}\taking\ell_2(b)\to j_{b'}$ for some $j_{b'}\in\Ob(I)$ with $\pi(j_b)=b'$. But now we have two maps in $I$ over the composite $F_1(b)\to F_2(b')$ both with source $\ell_1(b)\in\Ob(I)$, namely $\beta_{b'}\circ\ell_1(f)\taking\ell_1(b)\to\ell_2(b')$ and $h_{b'}\circ \beta_b\taking\ell_1(b)\to j_{b'}$. Since $\pi$ is a discrete opfibration, their codomains must be equal, so we have a map $\ell_2(f):=h_{b'}\taking\ell_2(b)\to \ell_2(b')=j_{b'}$, and we have completed the commutative square $X$ in Diagram (\ref{dia:crescent square}). We have now defined our functor $\ell_2\taking B\to I$ and natural transformation $\beta\taking\ell_1\to\ell_2$ over $\alpha$, and they are unique: we made no choices in their constructions. We have shown that $\pi^B\taking I^B\to S^B$ is a discrete opfibration. 

Let $\delta^B:=\del(\pi^B)\taking S^B\to\Set$ be the instance associated to $\pi^B$ and let $F\in\Ob(S^B)$ be an object. We can consider $F$ as a map $\ul{1}\To{F}S^B$, and $\delta^B(F)$ is isomorphic to the set of lifts in the left-hand diagram $$\xymatrix{&I^B\ar[d]^{\pi^B}\\\ul{1}\ar@{-->}[ur]\ar[r]_F&S^B}\hsp\hsp\xymatrix{&I\ar[d]^\pi\\B\ar@{-->}[ur]\ar[r]_F&S}$$ which by adjointness is in bijection with the set of lifts $\Gamma(F,\pi)$ in the right-hand diagram. Therefore we have $\delta^B(F)\iso\Gamma(F,\pi)$, completing the proof.

\end{proof}

The following definition of $\Prb(S)$ extends the notion of $\Prbs(S)$ from Definition \ref{def:strict morphisms}: $\Prbs(S)\ss\Prb(S)$ is a subcategory with the same set of objects. The category $\Prb(S)$ is a 2-category-theoretic version of slice categories, and is not new (see e.g. \cite{Kel}).

\begin{definition}\label{def:probes}

Let $S$ be a category. We define the {\em category of probes on $S$}, denoted $\Prb(S)$, as follows. \begin{align*}\Ob(\Prb(S))&=\{(A,F)\;|\; A\in\Ob(\Cat), F\taking A\to S \tn{ a functor}\}\\\Hom_{\Prb(S)}((A,F),(A',F'))&=\{G,\alpha)\;|\;G\taking A'\to A,\;\; \alpha\taking F\circ G\to F'\}\end{align*}$$\\\xymatrix{A'\ar[r]^G\ar@/_1.5pc/[rr]_{F'}&A\ar[r]^F\ar@{}[d]|(.35){\Down \alpha}&S\\~&~&~}$$ 

\end{definition}

\begin{remark}

In the presence of a discrete opfibration $\pi\taking I\to S$, a probe $F\taking A\to S$ sets up a {\em where-less query} on $\pi$ for which the results are the lifts $\ell\in\Gamma(F,\pi)$ for the diagram $$\xymatrix{\emptyset\ar[r]\ar[d]&I\ar[d]^\pi\\A\ar[r]_F\ar@{-->}[ru]^\ell&S.}$$ We call these where-less queries to emphasize that the where-category (upper left of the diagram) is empty.

\end{remark}

For any category $B$, there is an obvious functor $S^B\to\Prb(S)$. The following corollary extends Theorem \ref{thm:query discrete opfibration} in the obvious sense. One way to understand its content is that we can query over where-less queries. In other words, this is a formalization of nested queries. For example, we can create a join graph  of where-less queries and look for a set of coherent results. Corollary \ref{cor:new discrete opfibration from old} (which is not new) implies that given a morphism between two where-less queries on $S$ and given a result for the first query, there is an induced result for the second query. We will deal with the general case of nested queries (those having non-trivial where-categories) in Proposition \ref{prop:induced maps for queries}.

\begin{corollary}\label{cor:new discrete opfibration from old}

Let $\pi\taking I\to S$ be a discrete opfibration. Then the induced functor $$\ol{\pi}=\Prb(\pi))\taking\Prb(I)\to\Prb(S)$$ is a discrete opfibration. The instance associated to $\pi$ is $$\Gamma(-,\pi)=\del(\ol{\pi})\taking\Prb(S)\to\Set.$$

\end{corollary}

\begin{proof}

Proving this corollary is really just a matter of writing down the appropriate diagram. In order to show that $\ol{\pi}$ is a discrete opfibration, we choose an object $\ell\taking A\to I$ in $\Prb(I)$ with $\ol{\pi}(\ell)=F\taking A\to S$, we choose a morphism $(G,\alpha)\taking (A,F)\to (A',F')$ in $\Prb(S),$ and we show that there exists a unique morphism $(G,\beta)\taking(A,\ell)\to(A',\ell')$ in $\Prb(I)$, for some $\ell'\taking A'\to I$, such that $\pi\circ\beta=\alpha$. In diagrams, we begin with the solid-arrow diagram \begin{align}\label{dia:new discrete opfibration from old}\xymatrix{&&I\ar[d]^\pi\\A'\ar[r]^G\ar@/_1.5pc/[rr]_{F'}\ar@/^1pc/@{-->}[urr]^{\ell'}&A\ar@{}[u]|{\Uparrow\beta}\ar@{}[d]|(.35){\Downarrow\alpha}\ar[ur]^\ell\ar[r]^F&S\\~&~&~}\end{align} and hope to find such an $\ell'\taking A'\to I$ and $\beta\taking\ell\circ G\to\ell'$.

We have $\pi\circ(\ell\circ G)=F\circ G$. Applying Theorem \ref{thm:query discrete opfibration}, there is a unique induced functor $\ell'\taking A'\to I$ and natural transformation $\beta\taking\ell\to\ell'$ such that $\pi\circ\beta=\alpha$, having the required properties. This completes the proof.

\end{proof}

\begin{remark}

There is a way to express the set of solutions to a lifting problem using limits. Let $\pi\taking I\to S$ be a discrete opfibration, and consider the query $$\xymatrix{W\ar[r]^p\ar[d]_m&I\ar[d]^\pi\\R\ar[r]_n&S.}$$ We can consider $m$ as a strict morphism of probes on $S$, so it induces a function $\Gamma(m,\pi)\taking\Gamma(n,\pi)\to\Gamma(nm,\pi)$, and we can consider $p\in\Gamma(nm,\pi)$ as an element in the codomain. There is a bijection \begin{align}\label{dia:Gamma-m,p}\Gamma^{m,p}(n,\pi)\iso\Gamma(n,\pi)\cross_{\Gamma(nm,\pi)}\{p\},\end{align} expressing the set $\Gamma^{m,p}(n,\pi)$ of solutions to the lifting problem as the fiber of $\Gamma(m,\pi)$ over $p$. This idea may be useful when one has disjunctions in the WHERE-clause of a query, as one could replace $\{p\}$ with the set of disjuncts.

\end{remark}

Next we present examples of two types of morphisms of where-less queries, namely projection and indirection. These types generate all morphisms of where-less queries.

\begin{example}[Projection]\label{ex:projection}

Let $\delta\taking S\to\Set$ be an instance and let $\pi\taking I\to S$ be the associated discrete opfibration. Let $n\in\NN$ be a natural number. The {\em $n$-column table schema}, here denoted $C_n$, is the category with an initial object $K$, precisely $n$ other objects, and precisely $n$ non-identity arrows; it follows that $C_n$ looks like an asterisk (or ``star schema"), e.g. $C_4$ is drawn: 
$$\xymatrix{
&c_1\\c_4&K\ar[u]\ar[l]\ar[r]\ar[d]\ar[r]&c_2\\&c_3}$$ A functor $p\taking C_n\to S$ is called an {\em $n$-column table schema in $S$}.  For each object $x\in C_n$, we call $p(x)\in\Ob(S)$ a {\em column of $p$} and we call $p(K)$ the {\em primary key column of $p$}. In fact, $p$ is a probe or where-less query. The result set $\Gamma(p,\pi)$ can be thought of as the set of records for instance $\delta$ in table $p$; indeed $\Gamma(p,\pi)$ is isomorphic to $\delta(p)(K)$ as sets.

For any injection $h\taking\{1,2,\ldots,n'\}\inj\{1,2,\ldots,n\}$, there is an induced functor $C(h)\taking C_{n'}\to C_n$, which we can compose with $p$ to get a new morphism $p':=p\circ C(h)\taking C_{n'}\to S$ and a strict morphism of probes $p\to p'$. A record in table $p$ is given by a lift $\ell$ as shown to the left: 
\begin{align}\label{dia:higher queries}
\xymatrix{&&I\ar[d]^\pi\\C_{n'}\ar@/_1pc/[rr]_{p'}\ar[r]^{C(h)}&C_n\ar[ur]^\ell\ar[r]^p&S,}
\hspace{.7in}
\xymatrix{\fbox{$\LMO{\ell}$\hspace{.6in}}\ar[r]\ar[d]&\Prb(I)\ar[d]^{\Prb(\pi)}\\\fbox{$\LMO{p}\Too{C(h)}\LMO{p'}$}\ar[r]&\Prb(S)}
\end{align} 
and composing $\ell$ with $C(h)$ gives its projection as a record in table $p'$. Thus $h$ induces a function $\Gamma(p,\pi)\to\Gamma(p',\pi)$, and its image is the associated projection. The righthand diagram in (\ref{dia:higher queries}) is another way of viewing the lefthand diagram.

\end{example}

\begin{remark}

In Example \ref{ex:projection}, we did not really need to assume that the function $h$ was injective. If $h$ were not injective, then the morphism of queries $C(h)$ would result in some duplication of columns rather than a pure projection. In other words, the morphism of queries is simply given by substitution along the function $C(h)$.

\end{remark}

In Example \ref{ex:projection} we changed the shape of the result schema and used a strict morphism of probes (the natural transformation $p\circ C(h)\to p'$ was the identity). In Example \ref{ex:indirection} we will keep the result schema fixed but allow a non-strict morphism. 

\begin{example}[Indirection]\label{ex:indirection}

Let $R=[1]=\fbox{$\bullet^0\To{\;f\;}\bullet^1$}$ and let $S$ be the schema $$S:=\fbox{\xymatrix@=15pt{\obox{A}{.5in}{a person}\LA{rr}{lives at}&&\obox{B}{.7in}{an address}\LA{rr}{is in}&&\obox{C}{.4in}{a city}}}$$ There are three non-constant functors $R\to S$, which we denote $F_{AB}, F_{AC},$ and $F_{BC}$; there is a natural transformation $\alpha\taking F_{AB}\to F_{AC}$ and a natural transformation $\beta\taking F_{AC}\to F_{BC}$. Thus we get two morphisms in $\Prb(S)$, namely $$(\id_R,\alpha)\taking (R,F_{AB})\to (R,F_{AC})\hsp\tn{and}\hsp(\id_R,\beta)\taking (R,F_{AC})\to (R,F_{BC}).$$ 

Suppose $\pi\taking I\to S$ is an instance. We can draw the setup as 
$$\xymatrix{&&&I\ar[d]^\pi\\R\ar@{=}[r]\ar@/_3pc/[rrr]_(.3){F_{BC}}&R\ar@{}[d]|{\Down\beta}\ar@{=}[r]\ar@/_2pc/[rr]_(.35){F_{AC}}&R\ar@{}[d]|(.4){\Down\alpha}\ar[r]_{F_{AB}}&S\\&&&}
$$ 
We can take global sections $\Gamma(-,\pi)$ for each of these three probes and obtain maps between the result sets by Theorem \ref{thm:query discrete opfibration}: $$\Gamma(F_{AB},\pi)\To{\;\;\alpha\;\;}\Gamma(F_{AC},\pi)\To{\;\;\beta\;\;}\Gamma(F_{BC},\pi).$$ In other words, the morphism of queries induces a morphism of result sets. Simply, given some person and her address we can return a person and the city she lives in; given some person and his city we can return an address and the city it is in.

\end{example}

\subsection{The category of queries}\label{sec:category of queries}

We are now ready to generalize the category $\Prb(S)$ of where-less queries on $S$ to a category of all (lifting) queries on $S$.

\begin{definition}\label{def:cat of queries}

Let $\pi\taking I\to S$ denote a discrete opfibration. We define the {\em category of (lifting) queries on $\pi$}, denoted $\Qry(\pi)$ as follows. The objects of $\Qry(\pi)$ are commutative diagrams as to the left $$\xymatrix{&&\\W\ar[r]^p\ar[d]_m&I\ar[d]^\pi\\R\ar[r]_n&S\\&&}\hsp\hsp\xymatrix{&&\\W'\ar[r]_G\ar@/^1.4pc/[rr]^{p'}\ar[d]_{m'}&W\ar@{}[u]|(.3){\Uparrow\gamma}\ar[d]_m\ar[r]_p&I\ar[d]^\pi\\R'\ar[r]^F\ar@/_1.4pc/[rr]_{n'}&R\ar@{}[d]|(.3){\Down\alpha}\ar[r]^n&S\\&&}$$ and the morphisms $(F,G,\alpha,\gamma)\taking (R,W,n,p)\to (R',W',n',p')$ are diagrams as to the right, where $m\circ G=F\circ m'$ and where
$\pi\circ\gamma=\alpha\circ m'.$

\end{definition}

\begin{proposition}\label{prop:induced maps for queries}

Let $\pi\taking I\to S$ be a discrete opfibration, and suppose given the diagram to the left, where the two squares commute: 
$$\xymatrix{&&\\W'\ar[r]_G\ar[d]_{m'}&W\ar[d]_m\ar[r]_p&I\ar[d]^\pi\\R'\ar[r]^F\ar@/_1.4pc/[rr]_{n'}&R\ar@{}[d]|(.3){\Down\alpha}\ar[r]^n&S\\&&}
\hspace{.7in}
\xymatrix{&&\\W'\ar[r]_G\ar@/^1.4pc/[rr]^{p'}\ar[d]_{m'}&W\ar@{}[u]|(.3){\Uparrow\gamma}\ar[d]_m\ar[r]_p&I\ar[d]^\pi\\R'\ar[r]^F\ar@/_1.4pc/[rr]_{n'}&R\ar@{}[d]|(.3){\Down\alpha}\ar[r]^n&S\\&&}
$$
Then there exists a unique morphism of queries $$(F,G,\alpha,\gamma)\taking (R,W,n,,p)\to (R',W',n',p')$$ as to the right.

\end{proposition}

\begin{proof}

This is a direct application of Theorem \ref{thm:query discrete opfibration}. Indeed, in place of  Diagram (\ref{dia:natural transformation}), we draw $$\xymatrix{&&I\ar[dd]^\pi\\\\W'\ar[uurr]^{pG}\ar@/^1pc/[rr]^{nFm'}\ar@/_1pc/[rr]_{n'm'}\ar@{}[rr]|{\Down\alpha m'}&&S}$$ The unique functor and transformation labeled $\ell_2$ and $\beta$ given by the theorem serve as $p'$ and $\gamma$ here.

\end{proof}

\begin{theorem}

Let $\pi\taking I\to S$ be a discrete opfibration. Then $\Gamma^{-,-}(-,\pi)\taking\Qry(\pi)\to\Set$ is functorial. That is, given a morphism of queries $$\xymatrix{&&\\W'\ar[r]_G\ar@/^1.4pc/[rr]^{p'}\ar[d]_{m'}&W\ar@{}[u]|(.3){\Uparrow\gamma}\ar[d]_m\ar[r]_p&I\ar[d]^\pi\\R'\ar[r]^F\ar@/_1.4pc/[rr]_{n'}&R\ar@{}[d]|(.3){\Down\alpha}\ar[r]^n&S,\\&&}$$ there is an induced function, natural in $\Qry(\pi)$,  $$\Gamma^{m,p}(n,\pi)\too\Gamma^{m',p'}(n',\pi).$$

\end{theorem}

\begin{proof}[Sketch of proof]

Suppose given a lift $\ell\taking R\to I$ in $\Gamma^{m,p}(n,\pi)$. By Corollary \ref{cor:new discrete opfibration from old} we have a map $\ell'\taking R'\to I$, with $\pi\circ\ell'=n'$, and a natural transformation $\beta\taking\ell\circ F\to\ell'$, with $\pi\circ\beta=\alpha$. We need to show that $\ell'\circ m'=p'$ and $\beta\circ m'=\gamma$. But using the proof technique from Proposition \ref{prop:induced maps for queries}, this follows from Theorem \ref{thm:query discrete opfibration} and the definition of discrete opfibration.

\end{proof}

\begin{remark}

Given a discrete opfibration $\pi\taking I\to S$, we sometimes denote the functor $\Gamma^{-,-}(-,\pi)$ simply by $$\Gamma(\pi)\taking\Qry(\pi)\to\Set.$$

\end{remark}

\subsection{Data migration functors}\label{sec:migration}

Recall (from Definition \ref{def:migration}) that, given a functor $F\taking S\to T$, three data migration functors are induced between the categories $S\set$ and $T\set$. The most straightforward is denoted $\pullb{F}\taking T\set\to S\set$. It has both a left adjoint, denoted $\lpush{F}\taking S\set\to T\set$, and a right adjoint, denoted $\rpush{F}\taking S\set\to T\set$. 

In standard database contexts, schemas evolve over time. We model these schema evolutions as zigzags of functors from one schema to another, along which one can migrate data using a data migration functor. It is useful to know how this will affect queries. Typically, users of a database $\pi\taking I\to S$ are given access to a subset of $\Qry(\pi)$---they do not see the whole database, but instead some collection of queries. As the schema evolves it is important to understand how $\Qry(\pi)$ evolves. In this section we describe some results; for example under a pullback query results are unchanged. 

Let us begin by giving a description of $\rpush{F}$ in terms of where-less queries (see Section \ref{sec:where-less}). Recall that for any object $d\in\Ob(T)$ the ``comma" category $(d\down F)$ is defined as follows: \begin{align*}\Ob(d\down F)&=\{(c,f)\;|\;c\in\Ob(S), f\taking d\to F(c)\}\\\Hom_{(d\down F)}((c,f),(c',f'))&=\{g\taking c\to c'\;|\;f'\circ F(g)=f\}.\end{align*} There is a natural functor $n_d\taking (d\down F)\to S$, and given a morphism $h\taking d\to d'$ in $T$ we have a morphism $(d'\down F)\to(d\down F)$, or more precisely $n_{d'}\to n_d$, in $\Cat_{/S}$.

\begin{proposition}\label{prop:right push and fibrations}

Let $F\taking S\to T$ be a functor and $\gamma\taking S\to\Set$ an instance of $S$ with associated discrete opfibration $\pi\taking I\to S$. Given any object $d\in\Ob(T)$, there is an associated where-less query $$\xymatrix{&I\ar[d]^{\pi}\\(d\down F)\ar[r]_-{n_d}\ar@{-->}[ur]&S}$$ and we have $\rpush{F}(\gamma)(d)\iso\Gamma(n_d,\pi).$ Moreover, a morphism $d\to d'$ in $T$ induces a strict morphism of where-less queries $n_{d'}\to n_d$ ; thus we have a functor $T\to\Prbs(\pi)\op$. Then $\rpush{F}(\gamma)\taking T\to\Set$ is the composition $$T\To{d\mapsto n_d}\Prbs(\pi)\op\To{\Gamma(-,\pi)}\Set.$$

\end{proposition}

\begin{proof}

Let $F, \gamma, \pi, d,$ and $n_d\taking (d\down F)\to S$ be as in the proposition statement. By Proposition \ref{prop:where-less}, we have $\Gamma(n_d,\pi)\iso\lim_R(\gamma\circ n_d)$. This is exactly the formula for $\rpush{F}(\gamma)(d)$ by \cite[Theorem X.3.1]{Mac}, since $\rpush{F}$ is a right Kan extension. The statement for morphisms follows similarly.

\end{proof}

While Proposition \ref{prop:right push and fibrations} provides an interesting relationship between right pushforwards and queries, it does not allow us to relate queries on a database with queries on its right pushforward. In the following paragraphs, we will show briefly that graph pattern queries do transform nicely with respect to data migration functors $\lpush{F}$ and $\pullb{F}$.

We begin by discussing the left pushforward functor. Given a functor $F\taking S\to T$, we have a migration functor $\lpush{F}\taking S\set\to T\set$. 
%
If $\delta\in S\set$ and $\epsilon\in T\set$ are instances, then there is a bijection between the set of natural transformations $\lpush{F}\delta\to\epsilon$ and the set of commutative diagrams $$\xymatrix{\int(\delta)\ar[r]\ar[d]_{\pi_\delta}&\int(\epsilon)\ar[d]^{\pi_\epsilon}\\S\ar[r]_F&T.}$$ Given a query on $\pi_\delta$, we clearly obtain an induced query on $\pi_\epsilon$, and a solution to the former yields a solution to the latter: $$\xymatrix{W\ar[r]^p\ar[d]_m&\int(\delta)\ar[r]\ar[d]^{\pi_\delta}&\int(\epsilon)\ar[d]^{\pi_\epsilon}\\R\ar@{-->}[ur]\ar[r]_n&S\ar[r]_F&T.}$$ We state this formally in the following proposition.

\begin{proposition}\label{prop:query containment}

Let $F\taking S\to T$ be a functor, $\delta\in S\set$ and $\epsilon\in T\set$ instances, and $\lpush{F}\delta\to\epsilon$ a map of $T$-sets. There exists an induced functor of query categories and a natural transformation diagram: $$\xymatrix{\Qry(\pi_\delta)\ar[rr]\ar[dr]_{\Gamma(\pi_\delta)}&\ar@{}[d]|(.4){\Longrightarrow}&\Qry(\pi_\epsilon)\ar[dl]^{\Gamma(\pi_\epsilon)}\\&\Set&}$$

\end{proposition}

\begin{proof}

The proof follows from the discussion above.

\end{proof}

We now consider the case that $\delta\iso \pullb{F}\epsilon$. 

\begin{proposition}\label{prop:query iso}

Let $F\taking S\to T$ be a functor, let $\epsilon\taking T\to\Set$ be a functor, let $\delta=\pullb{F}\epsilon\taking S\to\Set$ be its pullback, and let $\pi_\delta$ and $\pi_\epsilon$ be as in Diagram (\ref{dia:query on pullback}) below. Then the results of any query on $\pi_\delta$ are the same as the results of the induced query on $\pi_\epsilon$. That is, we have a natural isomorphism diagram $$\xymatrix{\Qry(\pi_\delta)\ar[rr]\ar[dr]_{\Gamma(\pi_\delta)}&\ar@{}[d]|(.4){\parbox{.2in}{\begin{center}$\iso$\\\vspace{-.08in}$\Longrightarrow$\end{center}}}&\Qry(\pi_\epsilon)\ar[dl]^{\Gamma(\pi_\epsilon)}\\&\Set&}$$

\end{proposition}

\begin{proof}

Consider the diagram 
\begin{align}\label{dia:query on pullback}
\xymatrix{\int(\delta)\ar[r]\ar[d]_{\pi_\delta}\ullimit&\int(\epsilon)\ar[d]^{\pi_\epsilon}\\S\ar[r]_F&T,}
\end{align} 
which is a pullback by Proposition \ref{prop:pullback on Grothendieck}. Given a query on $\pi_\delta$, we obtain a query on $\pi_\epsilon$ as in Proposition \ref{prop:query containment}. The function from solutions for $\pi_\delta$ to solutions for $\pi_\epsilon$ is a bijection by the universal property of pullbacks. 
 $$\xymatrix{W\ar[r]^-p\ar[d]_m&\int(\delta)\ullimit\ar[r]\ar[d]_(.45){\pi_\delta}&\int(\epsilon)\ar[d]^{\pi_\epsilon}\\R\ar@/_.5pc/@{-->}[urr]\ar[r]_n&S\ar[r]_F&T.}$$ 
 Indeed, given a lift $R\to\int(\epsilon)$ of $\pi_\epsilon$, the fact that (\ref{dia:query on pullback} is a pullback means that there is a unique lift of $\pi_\delta$ mapping to it.

\end{proof}

\section{Future work}\label{sec:future}

This paper has set up an analogy between database queries and constraints on the one hand, and a now classical approach to algebraic topology---the lifting problem---on the other. Data on a schema is analogous to a covering space or fibration: the local quality of this fibration is determined by constraints, and the locating of sections that satisfy a set of properties is the posing of a query.

There are a few interesting directions for future research. The first is to make a connection to the relatively new field of {\em homotopy type theory (HoTT)} (see \cite{Awo},\cite{Voe}). The idea is that instead of two paths through a database schema being {\em equal}, one could declare them merely {\em equivalent}; if paths are declared equivalent in more than one way, these equivalences may also be declared as equivalent (or not). In this context, two observations on data may not be definitionally equal, but provably equal, and we consider the proofs and the differences between proofs as part of the data. To make this connection, the schema of a database should be a quasi-category (\cite{Jo2},\cite{Lur}) $\mcX$ rather than an ordinary category. Each higher simplex encodes a proof that different paths (or paths of paths, etc.) through the schema are equivalent. We might replace the instance data by a functor (map of quasi-categories) $\mcX\to\Type$, where $\Type$ is the quasi-category of homotopy types. In this context, classical homotopical questions, e.g. from the theory of model categories (\cite{Hir}) may be even more applicable.

Another direction for future research is to use topological tools to investigate or ``mine" data. For example, given a functor $\delta\taking S\to\Set$, we can compose with the functor $i\taking\Set\to\Top$ which sends each set to the corresponding discrete topological space. The homotopy colimit of $i\circ\delta$ is a topological space, of possibly any dimension and homotopy type, that encodes the connection pattern of the data. This space is homotopy equivalent to the nerve of the data bundle, $$\hocolim(i\circ\delta)\simeq N(\dispInt\delta)$$ (see \cite{Dug}). Thus we could report homotopy invariants of the data $\delta$, such as connected components, loops, etc. The question is whether these invariants would be meaningful and useful. For schemas of classical mathematical interest, such as the simplicial indexing category $S=\bD\op$, the homotopy colimit of $i\circ\delta$ is exactly what we want; it is the geometric realization of $\delta$. It remains to be seen whether such homotopy invariants may be useful in other contexts; e.g. there may be some connection to the analysis given by persistent homology (see \cite{Ghr},\cite{Car}).

A third and fairly straightforward project would be to adapt Garner's small object argument (see \cite{Gar}) to our notion of constraints. Garner's argument works, and provides nice universal properties, in the case of what we have called ``universal constraint sets" (see Section \ref{sec:discrete ops}). The question is, if we apply his techniques to local constraints, such as those in Example \ref{ex:const prod} used to declare that one table is the product of two others, does his procedure still result in a discrete opfibration with all the nice universal properties enjoyed in the universal case? We conjecture that it will. One should also check whether the results obtained from that procedure agree with those from the so-called universal chase procedure (see \cite{DNR}). Indeed, they should provide equivalent results, since both claim to be universal in the same way.

\bibliographystyle{amsalpha}

\end{document}